\documentclass[11pt,reqno]{amsart} 

\usepackage[top=30truemm,bottom=30truemm,left=30truemm,right=30truemm]{geometry}

\usepackage[dvipdfmx]{hyperref}
\usepackage{amscd}
\usepackage{amsmath, amssymb}
\usepackage{mathrsfs}
\usepackage{multicol}
\usepackage{ifthen}

\usepackage[all]{xy}
\usepackage{amsthm}
\usepackage[dvipsnames]{xcolor}
\usepackage{url}

\usepackage{paralist} 
\usepackage{enumerate}
\makeindex
\usepackage{latexsym}
\usepackage[dvipdfmx]{graphicx}
\usepackage{tikz}
\usetikzlibrary{intersections, calc}
\usepackage[stable]{footmisc}
\usepackage{xr}
\usepackage{nameref}
\usepackage{zref-xr}
\zxrsetup{toltxlabel} 

\usepackage[backend=bibtex, style=numeric, sorting=nty, giveninits=true, backref=false, backrefstyle=none, doi=false]{biblatex} 
\bibliography{mizuno_paper2.bib}
\AtBeginBibliography{\footnotesize}

\DeclareFieldFormat{biblabeldate}{\mkbibparens{\mkbibbold{#1}}}
\DeclareFieldFormat[article,periodical]{volume}{\mkbibbold{#1}}
\DeclareFieldFormat[article,periodical]{number}{\mkbibbold{#1}}
 
\hypersetup{
setpagesize=true, 
bookmarksnumbered=false, 
bookmarksopen=true, 
colorlinks=true,
linkcolor=RubineRed, 
citecolor=LimeGreen,}
\newtheorem{thm}{Theorem}[section]

\newtheorem{cor}[thm]{Corollary}
\newtheorem{lem}[thm]{Lemma}
\newtheorem{prop}[thm]{Proposition}
\newtheorem{cl}[thm]{Claim}

\theoremstyle{definition}
\newtheorem{dfn}[thm]{Definition}
\newtheorem*{dfn*}{Definition}
\newtheorem{nota}[thm]{Notation}

\newtheorem{eg}[thm]{Example}

\theoremstyle{remark}
\newtheorem{rmk}[thm]{Remark}
\newtheorem{egg}[thm]{Example}

\newcommand{\op}{\mathrm}

\numberwithin{equation}{subsection}

\makeatletter
\def\widebreve{\mathpalette\wide@breve}
\def\wide@breve#1#2{\sbox\z@{$#1#2$}%
     \mathop{\vbox{\m@th\ialign{##\crcr
\kern0.08em\brevefill#1{0.8\wd\z@}\crcr\noalign{\nointerlineskip}%
                    $\hss#1#2\hss$\crcr}}}\limits}
\def\brevefill#1#2{$\m@th\sbox\tw@{$#1($}%
  \hss\resizebox{#2}{\wd\tw@}{\rotatebox[origin=c]{90}{\upshape(}}\hss$}
\makeatletter

\begin{document}    
\newcounter{flag}

\title[Derived moduli stacks of Harder-Narasimhan filtrations]{An explicit construction of derived moduli stacks of Harder-Narasimhan filtrations}
\author[Yuki mizuno]{Yuki Mizuno}
\email{m7d5932a72xxgxo@fuji.waseda.jp}
\date{}
\address{Department~of~Mathematics, School~of~Science~and~Engineering, Waseda~University, Ohkubo~3-4-1, Shinjuku, Tokyo~169-8555, Japan}

\keywords{Moduli spaces of sheaves, Derived algebraic geometry, Harder-Narasimhan filtration}
\subjclass[2020]{14A30, 14D20, 14D23}

\maketitle

\begin{abstract}
In this article, we give an explicit construction of the derived moduli stack of Harder-Narasimhan filtrations on a smooth connected projective scheme over an algebraically closed field $k$ of characteristic $0$ by using the methods in \cite{behrend2014}. 
Moreover, we describe the derived deformation theory of a filtered sheave on a connected projective scheme over $k$ and compare our construction with the construction by Di Natale in \cite{di2017derived}.
\end{abstract}



\section{Introduction}
Moduli spaces of sheaves are constructed by using the methods of geometric invariant theory. 
In the constructions, they are constructed as quotients of open subschemes of Quot schemes that parametrize semistable sheaves.
The moduli schemes parametrize only semistable sheaves.
Unstable sheaves have unique filtrations described by semistable sheaves, which are called Harder-Narasimhan(HN) filtrations.
Although moduli spaces parametrizing unstable sheaves are not constructed as schemes, but as stacks. 
Moreover, moduli stacks of HN filtrations with fixed HN types are constructed as quotient stacks of open subschemes of relative flag schemes (for example, see \cite{yoshioka2009fourier}).

On the other hand, recent developments of derived algebraic geometry are remarkable.
As you can see in \cite{pantev2013shifted}, the theory of derived algebraic geometry provides a suitable framework for studying symplectic geometry of moduli stacks.
The authors proved that the derived moduli stack of perfect complexes on a smooth proper Calabi-Yau variety admits a shifted symplectic structure.

We are interested in  derived moduli stacks of HN-filtrations.
In particular, the aim of this paper is to give explicit constructions of derived moduli stacks of HN-filtrations as quotient derives stacks.
In the case of stable sheaves, explicit constructions of derived moduli spaces are studied in \cite{behrend2014} or \cite{borisov2020global}.
Moreover, explicit construction of derived Quot and Hilbert schemes are studied in \cite{ciocan2001} and \cite{ciocan2002derived}, respectively.

We use the methods in \cite{behrend2014}.
In the work, the authors explicitly describe the moduli space of semistable sheaves on a projective variety $X$ as an open substack of the stack of truncated graded $A$-modules, where $A := \oplus_{i \geq 0}\Gamma(X, \mathscr{O}_{X}(i))$. 
In detail, they use a functor from the category of coherent sheaves on $X$ to the category of $[p,q]$-graded $A$-modules defined by $ \Gamma_{[p,q]}(\mathscr{F}) := \oplus^q_{i=p} \Gamma(X, \mathscr{F}(i))$, where $\mathscr{F}$ is a coherent sheaf on $X$ and $p,q$ are nonnegative integers.
Then, a derived enhancement of the moduli of sheaves is constructed by using the structure of the derived moduli of graded $A$-modules from Hochschild cohomology. 
However, we need to have some obstacles in order to use their methods.
We need more detailed analyses of homological algebra of filtered (graded) modules which is developed by N\v{a}st\v{a}sescu and Van Oystaeyen \cite{nastasescu2006graded} and others.
 We also need to bridge between the deformation theory of filtered (graded) modules and that of filtered sheaves.
We have to study the relation between HN-filtrations of sheaves and of graded $A$-modules to get the explicit description.
For this, the work of Hoskins \cite{hoskins2018stratifications} is useful.
The author studies the relationship between HN-filtrations of quiver representations and of sheaves.
In Section 2, we deal with these problems and describe the image of the moduli stack of HN-filtrations with a fixed HN-type on $X$ inside the moduli stack of filtered $[p,q]$-graded $A$-modules explicitly (Theorem \ref{150320_6Jul21}). 
In Section 3, we construct a derived enhancement of moduli stacks of HN-filtrations and get the tangent complex at a point $[ 0  = \mathscr{F}^0 \subset \mathscr{F}^1 \subset \cdots \subset \mathscr{F}^s]$ of the derived stack of HN-filtrations on $X$ (Definition \ref{204924_6Jan22}, Theorem \ref{160827_17Jan22}).

In addition, the derived moduli spaces constructed in \cite{ciocan2001} or \cite{behrend2014} do not give the correct derived enhancement because higher cohomology of tangent complexes have discrepancy. 
This problem is explained in \cite[page 8]{borisov2020global}, \cite[page 5,6]{borisov2022shifted}, but not solved there.
In this paper, we fix it by using (inductive and projective) limits of finite dimensional dg-Lie algebras.
However, note that the derived moduli obtained from this construction is generally of infinite type.

In contrast to our approach, as a functorial approach, Di Natale \cite{di2017derived} constructs derived moduli stacks of filtered perfect complexes (on proper smooth schemes) which are locally geometric by using a model structure of the category of filtered complexes of modules over rings.
In Section 4, we translate our construction into theories of derived stacks and schemes in the sense of To\"en and compare our construction with that by Di Natale.
In \cite{alvarez2007functorial} or \cite{behrend2014}, functors between the category of Kronecker modules or graded modules over rings to the category of coherent sheaves on projective schemes are constructed. 
And, they study the relationship between moduli spaces of them.
For our purpose, we extend their argument to the theory of derived geometry, i.e. we construct functors between the category of (filtered) derived sheaves and the category of (filtered) $A_{\infty}$-modules and study them (Theorem \ref{232359_30Jan22}).
At the end of this section, we mention a simple example of Lagrangian morphisms related to derived moduli of HN-filtrations (Example \ref{132127_10Feb22}).

Moreover, our construction is useful to study details about symplectic geometry of derived moduli stacks of HN-filtrations because a theory for studying symplectic geometry of derived quotient stacks has been developed by Yeung (\cite{yeung2019precalabiyau}, \cite{yeung2021shifted}), recently.

\subsection*{Notation and conventions}
We work over an algebraically closed field $k$ of characteristic zero. For $m,n \in \mathbb{N}$, “ for $0 \ll m \ll n$ ” means $\exists m_0 \forall m \geq m_0 \exists n_0 \forall n \geq n_0$. 

\section{Preliminaries}

\subsection{Derived moduli schemes associated to bundles of curved differential graded Lie algebras}

\begin{dfn}(\cite{behrend2014})
A curved differential graded Lie algebra is a quadruple $(L^\bullet, f, d, [\cdot, \cdot])$, where $(L^\bullet, [\cdot, \cdot])$ is a $\mathbb{Z}_{\geq 0}-$graded Lie algebra, $f \in L^2$, and $d: L^\bullet \rightarrow L^\bullet$ is a degree one morphism of graded $k-$vector spaces such that
\begin{inparaenum}[(1)]
	\item $d(f) = 0$, 
	\item $d \circ d = [f, \cdot]$.
\end{inparaenum}

\end{dfn}

\begin{rmk}
$f$ is called the curving and $d$ the twisted differential.
	If $f = 0$, then a triple $(L^\bullet, d, [\cdot, \cdot])$ is called a differential graded Lie algebra (dgla).
\end{rmk}
 \begin{dfn}(\cite{behrend2014})
	Let $(L^\bullet, d, [\cdot, \cdot])$ be a dgla. Let $a \in L^1$.
	Then, $a$ is a Maurer–Cartan element if $a$ satisfies the equation
\begin{equation}
 da + \frac{1}{2}[a, a] = 0 .\notag
\end{equation}
	The set of Maurer-Cartan elements of $L$ is denoted by $\op{MC}(L)$.

\end{dfn}

\begin{thm}(\cite{behrend2014,behrend2018moduli})
	Let $X$ be a scheme over $k$. Let $\mathscr{L}$ be a bundle of curved differential graded Lie algebra (curved dgla) over $X$. Then, we can associate a sheaf of differential graded algebras $\mathscr{R}_X$ by letting the underlying sheaf of graded $\mathscr{O}_X-$algebra be
	\[ \mathscr{R}_X := \op{Sym}_{\mathscr{O}_X} \mathscr{L}[1]^\vee.
	\]
	The derivation $q$ on $\mathscr{R}_X$ is defined to be $q = q_1 + q_2 + q_3$, where $q_0: \mathscr{L}[1]^\vee \rightarrow \op{Sym}^0_{\mathscr{O}_X} \mathscr{L}[1]^\vee$ is defined by the curving morphism , $q_1: \mathscr{L}[1]^\vee \rightarrow \op{Sym}^1_{\mathscr{O}_X} \mathscr{L}[1]^\vee = \mathscr{L}[1]^\vee$ by the twisted differential and $q_2: \mathscr{L}[1]^\vee \rightarrow \op{Sym}^2_{\mathscr{O}_X} \mathscr{L}[1]^\vee$ by the bracket. 
	
 This defines a differential graded (dg) scheme $(\mathscr{R}_{X}, q)$(about dg-schemes, see \cite{ciocan2001}).
\label{205012_23Jun21}
 \end{thm}
 
\begin{eg}(\cite{behrend2014,behrend2018moduli})
 	Let $(L^\bullet, d, [\cdot, \cdot])$ be a dgla. $X := L^1 = \op{Spec}(\op{Sym}(L^{1\vee}))$. And, $\mathscr{L}^i \rightarrow L^1$ is the trivial vector bundle with fiber $L^i (i \geq 2)$. Then, $\{\mathscr{L}^i \}_{i \geq 2}$ have a structure of bundles of curved dgla over $L^1$ by 
 \begin{align*}
 		&\text{Curvatue}  &&f : L^1 \rightarrow \mathscr{L}^2, \quad a \mapsto (a, da+\frac{1}{2}[a,a]), \\
 		&\text{Differential} &&{d}' : \mathscr{L}^i \rightarrow \mathscr{L}^{i+1}, \quad (\mu, b) \mapsto (\mu, db+[\mu, b]), \\
 		&\text{Bracket} &&[\cdot, \cdot]' : \mathscr{L}^i \times \mathscr{L}^j \rightarrow \mathscr{L}^{i+j}, \quad ( (\mu, b), (\lambda, c) ) \mapsto (\mu+\lambda, [b, c]).
 \end{align*}
 	In addtion, we have an isomorphism $\op{Spec}(H^0(\mathscr{R}_X)) \simeq \op{MC}(L) = Z(f)$.
\label{205047_23Jun21}
 \end{eg}

\subsection{Homological algebra of filtered modules}
\label{153759_23Jun21} 

Let $R$ be a unital commutative ring
\begin{dfn}
(\cite{illusie1971}, \cite{kremer2014}, \cite{nastasescu2006graded})
A filtered $R$-module $M$ is a $R$-module $M$ with a ascending chains $\{M^i \mid i \in \mathbb{Z}\}$ of $R$-submodule of $M$ such that if $i \leq 0$, $M^{i} = 0$ and for $i \gg 0$, $M^{i} = M$.
When  we denote the minimum integer of the integers by $s$, 
we write
	\[ 0 = M^0 \subset M^1 \subset \cdots \subset M^s = M. \]
	And, let $M_1, M_2$ be filtered $R-$modules. 
	A homomorphism of filtered $R$-modules from $M_1$ to $M_2$ is a $R-$module homomorphism from $M_1$ to $M_2$ preserving their filtration. We denote the set of homomorphisms by $\op{Hom}_{R,-}(M_1, M_2)$. 
\label{160543_25Jun21}
\end{dfn}

\begin{dfn}
(\cite{illusie1971}, \cite{kremer2014}, \cite{nastasescu2006graded})
\label{160551_25Jun21}
	Let $M$ be a filtered $R$-module. Then, associated graded module $\op{gr}(M)$ is defined to be
	\[ \op{gr}(M) := \oplus M^i/M^{i-1} \quad \op{gr}_i(M):=  M^i/M^{i-1} .\] 
\end{dfn}

\begin{dfn}
(\cite{illusie1971}, \cite{kremer2014}, \cite{nastasescu2006graded})
	Let $I$ be a filtered $R$-module. We say $I$ is a filtered injective if the components $\op{gr}_i(I)$ of the associated graded module are injective $R$-modules.
\label{160607_25Jun21}
\end{dfn}
\begin{dfn}
(\cite{illusie1971}, \cite{kremer2014}, \cite{nastasescu2006graded})
\label{160617_25Jun21}
	Let $M$ be a filtered $R$-module. Then, a filtered injective resolution $M$ is a exact sequence of $R$-modules
	\[ 0 \rightarrow M \rightarrow I^0 \rightarrow I^1 \rightarrow \cdots \]
	such that induced sequences
	\[ 0 \rightarrow \op{gr}_i(M) \rightarrow \op{gr}_i(I^0) \rightarrow \op{gr}_i(I^1) \cdots  \]
	are the injective resolutions of $\op{gr}_i(M)$, where $I^0, I^1, \cdots $ be filtered injective $R$-modules.
\end{dfn}

\begin{dfn}
(\cite{drezet1985, huybrechts2010geometry},)
\label{161700_25Jun21}
	Let $M,N$ be a filtered $R$-module. Let $0 \rightarrow M \rightarrow I^{\bullet}$ be a filtered injective resolution of $R$. Then, 
	\[ \op{Ext}_{R,-}^i (N, M) : = H^i(\op{Hom}_{R,-}(N, I^\bullet)). \]
\end{dfn}

\begin{rmk}
	In the above definition, $\op{Ext}_{R,-}(N,M)$ are independent of the choice of the filtered resolution of $M$.
	In \cite{drezet1985, huybrechts2010geometry}, $\op{Ext}_-^i(N,M)$ are defined when $M,N$ are filtered coherent sheaves on algebraic varieties.
\end{rmk}

\begin{thm}
(\cite{drezet1985,huybrechts2010geometry})
\label{161710_25Jun21}
	Let $M,N$ be filtered $R$-modules. There are spectral sequences.	
\[ \op{Ext}_{R, -}^{p+q}(N,M) \Leftarrow E_1^{pq} = \begin{cases}
0 & p < 0 \\
\prod_i \op{Ext}_{R}^{p+q}(\op{gr}_i(N), \op{gr}_{i-p}(M))  & p \geq 0
\end{cases}
	\]
\end{thm}

\subsection{Harder-Narasimhan filtration of sheaves and modules}
\subsubsection{For sheaves}
Let $X$ be a projective $k$-scheme of finite type and $\mathcal{O}_X(1)$ be a very ample invertible sheaf on $X$.
\begin{dfn}
	Let $P(t), Q(t) \in \mathbb{Q}[t]$. We say $P(t) (\succeq) Q(t)$ if 
	 \[\frac{P(m)}{P(n)} (\geq) \frac{Q(m)}{Q(n)}  \quad \text{for} \quad m \gg n \gg 0.\]
\end{dfn}
\begin{dfn}
(\cite{behrend2014,hoskins2018stratifications})
Let $\mathscr{F}$ be a coherent sheaf on $X$.
$\mathscr{F}$ is (semi)stable if for every proper nonzero subsheaf $\mathscr{E}$, $P(\mathscr{F}, t) (\succeq) P(\mathscr{E}, t)$, where $P(\mathscr{F}, t)$ and $P(\mathscr{E}, t)$ are the Hilbert polynomials of $\mathscr{F}$ and $\mathscr{E}$ with respect to $\mathcal{O}_X(1)$ respectively. \label{def:stability}
\end{dfn}
\begin{dfn}
(\cite{hoskins2018stratifications})
	Let $\mathscr{F}$ be a coherent sheaf on $X$. The Harder-Narasimhan(HN) filtration of $\mathscr{F}$ is a filtration 
\[
0= \mathscr{F}^0 \subset \mathscr{F}^1 \subset \cdots \subset \mathscr{F}^s = \mathscr{F}
\]
such that the $ \mathscr{F}^i/\mathscr{F}^{i-1}$ are semistable and $P(\mathscr{F}^1/ \mathscr{F}^0, t) \succ P(\mathscr{F}^2/ \mathscr{F}^1, t) \succ \cdots \succ P(\mathscr{F}^s/ \mathscr{F}^{s-1}, t)$. \label{def:HN-filtration}
\end{dfn}
\begin{thm}
(\cite{hoskins2018stratifications})
	Every coherent sheaf on $X$ has a unique HN-filtration \label{thm:HN-filtration}
\end{thm}
\begin{rmk}
	Our definitions of stability and HN-filtration are different from those in \cite{huybrechts2010geometry}. However, these definitions coincide when $\mathcal{F}$ is a pure sheaf. A benefit of our definitions is that we can consider all coherent sheaves within torsion sheaves.
\end{rmk}

\subsubsection{For modules}
Let $A$ be a graded $k$-algebra which is all in non negative degrees and each graded piece is finite-dimensional and $A_0 = k$.

\begin{dfn}
(\cite{behrend2014})
 Let $M$ be a $[p,q]$-graded $k$-module.
 Let $\lambda : A \otimes_k M \rightarrow M$ be a homomorphism of graded $k$-vector space.
Then, we call $M$ a $\lambda$-module and a $k$-submodule $N$ of $M$ is a $\lambda$-submodule of $M$ if $\lambda(A \otimes_k N ) \subset N$.
\end{dfn}

\begin{rmk}
\begin{itemize}
\item In the above definition, $\lambda$ is determined by $\lambda|_{A_{[0, q-p]} \otimes_k M}$
 \item  Any $[p,q]$-graded $A$-module $M$ has a natural $\lambda$-module structure from the $A$-module structure of $M$. 
\end{itemize}
\end{rmk}
\begin{dfn}
	Let $M$ be a $[p,q]$-graded $k$-module with $\op{dim}M_p + \op{dim}M_q \neq 0$. Let $\theta_p, \theta_q $ be integers, then 
	\[ \mu_{(\theta_p, \theta q)}(M):= \frac{\theta_p\op{dim}M_p + \theta_q \op{dim}M_q}{\op{dim}M_p + \op{dim}M_q} .\]
	
	We write $\mu$ for $\mu_{(\theta_p, \theta_q)}$ when $\theta_p, \theta_q$ are obvious.
\end{dfn}
\begin{dfn}
(\cite{behrend2014})
	Let $M$ be a $[p,q]$-graded $\lambda$-module. 
Let $\theta_p, \theta_q$ be integers. 
 $M$ is (semi)stable with respect to $(\theta_p, \theta_q)$ if for every nonzero proper $\lambda$-submodule $N$,  $\op{dim}N_p = \op{dim}N_q = 0$ or “$\op{dim}N_p + \op{dim}N_q \neq 0$ and $\mu_{(\theta_p,\theta_q)}(N) (\leq) \mu_{(\theta_p,\theta_q)}(M)$” holds.
\end{dfn}
\begin{dfn}

	Let $M$ be a $[p,q]$-graded $\lambda$-module. Let $\theta_p := \op{dim}M_q$ and $ \theta_q := -\op{dim}M_p$. The Harder-Narasimhan(HN) filtration of $M$ is a filtration of $\lambda$-submodules of $M$
	\[
	0 = M^0 \subset M^1 \subset \cdots \subset M^s = M
	\]
	such that the $M^i/M^{i-1}$ are semistable with respect to $(\theta_p, \theta_q)$ and $(M^1/M^0) \succ (M^2/M^1) \succ \cdots \succ (M^s/M^{s-1})$.
	
	 For two $[p,q]$-graded $\lambda$-modules $N^1,N^2$, we say $N^1 (\succeq) N^2$ if $\op{dim}N^2_p = \op{dim}N^2_q = 0$ or “$\op{dim}N^1_p + \op{dim}N^1_q \neq 0$ and  $\op{dim}N^2_p + \op{dim}N^2_q \neq 0$ and $\mu_{(\theta_p, \theta_q)}(N^1) (\geq) \mu_{(\theta_p, \theta_q)}(N^2)$”. 
	 \label{def:module-hn}
\end{dfn}

 \begin{rmk}
  The above definition is similar to that of quiver representation, but slightly different from it. 
The difference derives from the existence of the components $M_{p+1}, \cdots M_{q-1}$ (cf. \cite{reineke2003harder}, \cite{zamora2014harder} or \cite{hoskins2018stratifications})
 \end{rmk}

\begin{thm}Let $M$ be a $[p,q]$-graded $\lambda$-module.
There exists a HN-filtration of $M$	
\end{thm}

\begin{proof}
	We can prove this in the same way as in the proof of \cite[Proposition 2.5]{reineke2003harder} or \cite[Theorem 2.6]{zamora2014harder}
\end{proof}

\begin{rmk}
	 Every $[p,q]$-graded $\lambda$-module dose not necessarily have a unique HN-filtration because the relation $\succ$ in Definition \ref{def:module-hn} is not a stability structure on \cite[Def 1.1]{rudakov1997stability} (i.e., the seesaw property does not hold).
However, note that this does not have negative effects on this article.
\end{rmk}

\section{Moduli stacks of Harder-Narasimhan filtrations}

\begin{nota}
\begin{itemize}
 \item  $X:$ a connected projective scheme over $k$.
 \item  $\mathscr{O}_X(1):$ a very ample imbertible sheaf on $X$.
 \item  $A := \Gamma_*(\mathscr{O}_X) = \oplus_{i \geq  0} \Gamma(X,\mathscr{O}(i))$.
\item $\mathfrak{m} := A_{> 0 }$.
\item $\mathcal{C}\op{oh}_{\alpha}(X)$ : the stack of coherent sheaves with Hilbert polynomial $\alpha$ on $X$.
 \item $\mathcal{FC}\op{oh}_{(\alpha_1, \cdots, \alpha_s)}(X)$ : the stack of filtered coherent sheaves on $X$ of type $(\alpha_1, \cdots, \alpha_s)$
 \item $\mathcal{FC}\op{oh}^{\text{HN}}_{(\alpha_1, \cdots, \alpha_s)}(X)$ : the stack of Harder-Narasimhan filtrations (of sheaves) of type $(\alpha_1, \cdots, \alpha_s)$ on $X$.
 \item $\mathcal{M}\op{od}^{[p,q]}_{\alpha}(A)$ : the stack of graded $A$-modules of type $\alpha|_{[p,q]}$ in degree $[p,q]$. 
 \item $\mathcal{FM}\op{od}^{[p,q]}_{(\alpha_1, \cdots , \alpha_s)}(A)$ : the stack of filtered graded $A$-modules of type $(\alpha_1|_{[p,q]}, \cdots , \alpha_s|_{[p,q]})$ in degree $[p,q]$, 
where $\alpha_i|_{[p,q]}$ means a tuple $(\alpha_i(p), \cdots , \alpha_i(q))$.
\item $\mathcal{FM}\op{od}^{[p,q], \text{sfg}}_{(\alpha_1, \cdots , \alpha_s)}(A)$ : the stack of strongly finitely generated filtered graded $A$-modules of type $(\alpha_1|_{[p,q]}, \cdots , \alpha_s|_{[p,q]})$ in degree $[p,q]$.
\item $\mathcal{FM}\op{od}^{[p,q]_{\text{HN}}, \text{sfg}}_{(\alpha_1, \cdots , \alpha_s)}(A)$ : the stack of strongly finitely generated filtered graded $A$-modules of type $(\alpha_1|_{[p,q]}, \cdots , \alpha_s|_{[p,q]})$ in degree $[p,q]$ which are HN-filtrations for $(\alpha(q), -\alpha(p))$-stability.
\item $\op{Coh}(Y)$ : the category of coherent sheaves on a scheme $Y$.
\item $\op{FCoh}(Y)$ : the category of filtered coherent sheaves on a scheme $Y$.
\item $\op{Mod}^{[p,q]}(A \otimes \mathscr{O}_{Y})$ : the category of graded coherent $A \otimes \mathscr{O}_Y$-modules on $Y$ in the degree $[p,q]$.
\item $\op{FMod}^{[p,q]}(A \otimes \mathscr{O}_{Y})$ : the category of filtered graded coherent $A \otimes \mathscr{O}_Y$-modules on $Y$ in the degree $[p,q]$.
\item $\Gamma_{[p,q]}(\mathscr{F}) := \oplus^q_{i=p} \pi_*(\mathscr{F}(i)) $ for a coherent sheaf $\mathscr{F} $ on $X \times_k Y$ and the natural projection $\pi: X \times_k Y \rightarrow Y $.
\end{itemize}
\end{nota}
\begin{rmk}
 When considering any filtered $[p,q]$-graded $A$-modules $0= M^0 \subset M^1 \subset \cdots \subset M^s = M$, it is strongly finitely generated if each $M^i$ is graded $A$-module which is generated in degree $p$ .
\end{rmk}

\subsection{Open embeddings}
 
In this subsection, we construct an open immersion from the moduli stack of filtered sheaves to that of filtered modules over $A$

First, we can define the following  morphism 
\begin{equation}
 \Gamma^{\op{fil}}_{[p,q]}: \op{FCoh}(X \times_k Y) \rightarrow \op{FMod}^{[p,q]}(A \otimes \mathscr{O}_Y)
\end{equation}
by $\Gamma^{\text{fil}}_{[p,q]}(\mathscr{F}):= \Gamma_{[p,q]}(\mathscr{F)} = \Gamma_{[p,q]}(\mathscr{F}^s) \supset \cdots \supset \Gamma_{[p,q]}(\mathscr{F}^1) \supset \Gamma_{[p,q]}(\mathscr{F}^0) = 0$ for any object $\mathscr{F}= \mathscr{F}^s \supset \cdots \supset \mathscr{F}^1 \supset \mathscr{F}^0=0$ because pushforwards are left exact.

 The morphism of category
\[
  \Gamma_{[p,q]}: \op{Coh}(X \times_k Y) \rightarrow \op{Mod}^{[p,q]}(A \otimes \mathscr{O}_Y)
\]
has the left adjoint $\mathscr{S}$ (see \cite[Proposition 3.1]{behrend2014}).



\subsubsection{Monomorphisms}
\begin{lem}
 $\Gamma^{\op{fil}}_{[p,q]}|_{ \mathcal{FC}\op{oh}^{\op{HN}}_{(\alpha_1, \cdots, \alpha_s)}(X)} : \mathcal{FC}\op{oh}^{\op{HN}}_{(\alpha_1, \cdots, \alpha_s)}(X) \rightarrow \mathcal{FM}\op{od}^{[p,q]}_{(\alpha_1, \cdots , \alpha_s)}(A) $ is monomorphism if $q \gg p \gg 0$.
\label{224726_27Jun21}
\end{lem}
\begin{proof}
First, let $S := \{\mathscr{F} \in \op{Coh}(X) \mid \text{the HN-type of } \mathscr{F} \text{ is }(\alpha_1, \cdots \alpha_s)\}$.
Then, the sets $S_i = \{\op{gr}_i(\mathscr{F}) \mid \mathscr{F} \in S \} (i \in \{1,\cdots, s\})$ are bounded because the stack of HN-filtrations of type $(\alpha_1, \cdots, \alpha_s)$ is a quotient of a relative flag scheme by an algebrac group (see \cite[Lemma 2.5]{yoshioka2009fourier}).
So, for $p \gg 0$, any sheaf in $\cup_i S_i$ is $p$-regular, and $\Gamma^{\op{fil}}_{[p,q]}|_{ \mathcal{FC}\op{oh}^{\op{HN}}_{(\alpha_1, \cdots, \alpha_s)}(X)}$ is well-defined.

In order to prove $\Gamma^{\op{fil}}_{[p,q]}|_{ \mathcal{FC}\op{oh}^{\op{HN}}_{(\alpha_1, \cdots, \alpha_s)}(X)}$ is a monomorphism, we will prove this is fully faithful because monomorphisms between algebraic stacks are the same as fully faithful functors (\cite[Lemma 98.8.4]{stacks-project}).

 If $\mathscr{F},\mathscr{G}$ are objects of $ \mathcal{FC}\op{oh}^{\op{HN}}_{(\alpha_1, \cdots, \alpha_s)}(X)$, we have an isomorphism
\[
\op{Hom}_{\mathscr{O}_{X \times Y}}(\mathscr{F}, \mathscr{G}) \overset{\simeq}{\rightarrow} \op{Hom}_{\op{gr}, A \otimes\mathscr{O}_Y}(\Gamma_{[p,q]}(\mathscr{F}),\Gamma_{[p,q]}(\mathscr{G}))  \quad q \gg p\gg 0
 \] 
because the induced morphism $\Gamma_{[p,q]}: \mathcal{U}_{\alpha} \rightarrow \mathcal{M}\op{od}^{[p,q]}_{\alpha}(A)$ is a monomorphism for $q \gg p \gg 0$ (\cite[Proposition 3.2]{behrend2014}), where $\mathcal{U}_{\alpha}$ is an open substack of $\mathcal{C}\op{oh}_{\alpha}(X)$ of finite type.
 So, $\Gamma^{\text{fil}}_{[p,q]}$ is faithful.

 Next we show that $\Gamma^{\text{fil}}_{[p,q]}$ is full. 
We take integers $p,q$ such that $\Gamma_{[p,q]}: \mathcal{U}_{\alpha_i} \rightarrow \mathcal{M}\op{od}^{[p,q]}_{\alpha_i}(A)$ are monomorphism.
Let $\psi: \Gamma^\text{fil}_{[p,q]}(\mathscr{F} ) \rightarrow \Gamma^\text{fil}_{[p,q]}(\mathscr{G})$ be a morphism in $\mathcal{FM}\op{od}^{[p,q]}(A)$. 
From the above isomorphism, there exists a morphism $\varphi : \mathscr{F} \rightarrow \mathscr{G}$ such that $\Gamma_{[p,q]}(\varphi) = \psi$.
Then, it is sufficient to see that $\varphi$ is a morphism in $\mathcal{FC}\op{oh}(X)^{\op{HN}}_{(\alpha_1, \cdots , \alpha_s)}$.
We may assume that $s = 2$. And, we consider the following diagrams and the corresponce by $\Gamma_{[p,q]}$ and $\mathscr{S}$:
{\small
\begin{align*}
\def\g#1{\save
 [].[d]!C="g#1"*[F(]\frm{}\restore}%
\def\h#1{\save
 [].[d]!C="h#1"*[F)]\frm{}\restore}%
\xymatrix@C=35pt@M=5pt{
 \g1 \mathscr{F}^2 \ar[r]^{\varphi} 
& \h1  \mathscr{G}^2 
& \g2  \Gamma_{[p,q]}\mathscr{F}^2 \ar[r]^{\psi} 
& \h2 \Gamma{[p,q]}\mathscr{G}^2 
& \g3 \mathscr{S}\Gamma_{[p,q]}\mathscr{F}^2 \ar[r]^{\mathscr{S}\psi} 
& \h3 \mathscr{S}\Gamma_{[p,q]}\mathscr{G}^2 \\
\mathscr{F}^1 \ar[r]_{\varphi'} \ar@{^{(}->}[u]_{i_1} 
& \mathscr{G}^1 \ar@{^{(}->}[u]^{i_2} 
 &  \Gamma_{[p,q]}\mathscr{F}^1 \ar@<-1ex>[r]_{\psi|_{\Gamma_{[p,q]} \mathscr{F}^1} } \ar@<3ex>@{^{(}->}[u]_{\Gamma_{[p,q]}i_1} 
& \Gamma_{[p,q]}\mathscr{G}^1 \ar@<-3ex>@{^{(}->}[u]^{\Gamma_{[p,q]}i_2}  \ar@{}[lu]|{\circlearrowright} 
&  \mathscr{S}\Gamma_{[p,q]}\mathscr{F}^1 \ar@<-1ex>[r]_{\mathscr{S}\psi|_{\Gamma_{[p,q]} \mathscr{F}^1}} \ar@<3ex>[u]_{\mathscr{S}\Gamma_{[p,q]}i_1} 
&  \mathscr{S}\Gamma_{[p,q]}\mathscr{G}^1 \ar@<-3ex>[u]^{\mathscr{S}\Gamma_{[p,q]}i_2} \ar@{}[lu]|{\circlearrowright}
\ar@{} "h1" ;  "g2"^(.20){}="c"^(.60){}="d"
\ar@{|->} "c";"d"^-{\Gamma_{[p,q]}} 
\ar@{} "h2" ;  "g3"^(.35){}="e"^(.60){}="f"
\ar@{|->} "e";"f"^-{\mathscr{S}} 
}
\end{align*}
}
, where $\varphi$ and $\varphi'$ are the morphisms corresponding to $\psi$ and $\psi|_{\Gamma_{[p,q]} \mathscr{F}^1}$ respectively. 
And, the middle and the right squares are commutative.
A diagram chase and the monomorphisity of $\Gamma_{[p,q]}$ yield the commutativity of the left diagram above because that $\Gamma_{[p,q]}$ is fully faithful and $\mathscr{S}$ is a left adjoint to $\Gamma_{[p,q]}$ is equivalnt to that $\mathscr{S} \circ \Gamma_{[p,q]} \simeq \op{id}$ (\cite[Lemma 4.24.4]{stacks-project}).
 
\end{proof}



\subsubsection{\'Etale morphisms}

In this subsection, we prove that the morphism of algebraic stacks 
\[\Gamma^{\op{fil}}_{[p,q]} : \mathcal{FC}\op{oh}_{(\alpha_1, \cdots, \alpha_s)}(X) \rightarrow \mathcal{FM}\op{od}^{[p,q]}_{(\alpha_1, \cdots , \alpha_s)}(A)
\]
is \'etale. 
To prove this, we need to consider the deformation theories of $\mathcal{FC}\op{oh}_{(\alpha_1, \cdots, \alpha_s)}(X) $ and $\mathcal{FM}\op{od}^{[p,q]}_{(\alpha_1, \cdots , \alpha_s)}(A)$ and compare them.

First, we consider $\mathcal{FM}\op{od}^{[p,q]}_{(\alpha_1, \cdots , \alpha_s)}(A)$. $0=V^0 \subset V^1 \subset \cdots \subset V^s = V$ is a filtration of $\mathbb{N}$-graded $k$-vector spaces such that $\op{dim}_k V^i_j = \alpha_i(j)$, where $V^i := \oplus_j V^{i}_j $.
Then, 
\[
L_{[p,q]} := \bigoplus^{\infty}_{n=0} L^n_{[p,q]} := \bigoplus^{\infty}_{n=0} \op{Hom}_{k-\op{gr}}(\mathfrak{m}^{\otimes n}, \op{End}_{k}(V_{[p,q]})) = \op{Hom}_{k-\op{gr}}(\mathfrak{m}^{\otimes n}\otimes V_{[p,q]}, V_{[p,q]}  ) 
     \]
has a dgla structure by the Hochschild differential $d$ and the Gerstenhaber bracket $[\cdot,\cdot]$ (in detail, see \cite{behrend2014} or \cite{behrend2018moduli}). Note that when let $G_{[p,q]} := \op{GL}_{\op{gr}}(V_{[p,q]})$ (called the gauge group), we have an action of $G_{[p,q]}$ on $L_{[p,q]}$ called the gauge action.  However, in our situation, we need to consider the filtration of $V_{[p,q]}$. So, we consider the following graded vector space
\[
 L_{[p,q],-} := \bigoplus^{\infty}_{n=0} L^n_{[p,q],-} := \bigoplus^{\infty}_{n=0} \op{Hom}_{k-\op{gr}}(\mathfrak{m}^{\otimes n}, \op{End}_{k,-}(V_{[p,q]})) = \bigoplus^{\infty}_{n=0} \op{Hom}_{k-\op{gr}, -}(\mathfrak{m}^{\otimes n}\otimes V_{[p,q]}, V_{[p,q]} )
\]
(for the above notation, see subsection \ref{153759_23Jun21}). Note that $L_{[p,q],-}$ is a graded $k$-vector subspace of $L_{[p,q]}$. 
We can easily see Hochschild differential and the Gerstenhaber bracket of $L_{[p,q]}$ is closed in $L_{[p,q],-}$. 
So, $(L_{[p,q],-}, d, [\cdot,\cdot])$ is also a dgla. We also consider the parabolic subgroup $P_{[p,q]}$ of $G_{[p,q]}$ consisting of elements preserving the filtration of $V_{[p,q]}$. This also induces an action on $L_{[p,q],-}$. 
And, we have $\op{MC}(L_{[p,q],-}) = \{\mu \in \op{Hom}_{k\op{-gr}, -}(V_{[p,q]},V_{[p,q]}) \mid  \mu \text{ induces a filtered graded $A$-module structure }\} $. 
So, we have 
\begin{equation}
\label{163045_6Jul21}
[\op{MC}(L_{[p,q],-})/P_{[p,q]}] \simeq  \mathcal{FM}\op{od}^{[p,q]}_{(\alpha_1, \cdots , \alpha_s)}(A) 
\end{equation}
Moreover, since $L_{[p,q],-}$ is a dgla, we get a differential graded structure on $L_{[p,q],-}$ by using Theorem \ref{205012_23Jun21} and Example \ref{205047_23Jun21}. 
Because the Lie algebra of $P_{[p,q]}$ is $\op{Hom}_{k-\op{gr},-}(V_{[p,q]},V_{[p,q]})$, we can describe the tangent and obstruction spaces of a filtered graded $A$-module $(V_{[p,q]}, \mu)$ by describing the deformation theory of $[\op{MC}(L{[p,q],-})/P_{[p,q]}]$.
I.e., 
\begin{align*}
&\text{The infinitesimal deforamtion of $(V_{[p,q]}, \mu)$ is classified by } H^1((L_{[p,q],-}, d^\mu, [\cdot, \cdot])), \\
&\text{The obstructions of the deformation of $(V_{[p,q]}, \mu)$ are contained in } H^2((L_{[p,q],-}, d^\mu, [\cdot,\cdot]))
\end{align*}
, where $(L_{[p,q],-}, d^\mu, [\cdot, \cdot])$ is a dgla whose differential $d^{\mu}$ is defined by $d + [\mu, \cdot]$ (for detail, see \cite{behrend2014} or \cite{behrend2018moduli}). 
Although the right-hand sides of the above equalities are equal to the Hochschild cohomology, we need a lemma to associate them with $\op{Ext}_{A-\op{gr},-}(-, -)$ functor (see Remark \ref{205937_25Jun21}).

\begin{dfn}
(for the case of ungraded filtered modules, see \cite{illusie1971}, \cite{kremer2014}, \cite{nastasescu2006graded})
 Let $P$ be a filtered graded $A$-module. Then, $P$ is filtered projective if $\op{gr}_i(P)$ is projective object in the category of graded $A$-modules for any $i$.

Let $M$ be a filtered graded $A$-module. Then, a filtered graded projective resolution of $M$ is a sequence 
\[
 \cdots \rightarrow P^1 \rightarrow P^0 \rightarrow M \rightarrow 0
\]
such that the induced sequences 
\[
 \cdots \rightarrow \op{gr}_i(P^1) \rightarrow \op{gr}_i(P^0) \rightarrow \op{gr}_i(M) \rightarrow 0
\]
are projective resolutions of $\op{gr}_i(M)$ in the category of graded $A$-module.
\end{dfn}
\begin{rmk}
 We have graded versions of Definition \ref{160543_25Jun21}, \ref{160607_25Jun21}, \ref{160617_25Jun21} as above. And, we can define a graded version $\op{Ext}^i_{A-\op{gr}, -}(-, -)$ of the filtered Ext functor $\op{Ext}^i_{A-gr,-}(-, -)$ in Definition \ref{161700_25Jun21} and have a graded version of Theorem \ref{161710_25Jun21}.
\label{205937_25Jun21}
\end{rmk}

\begin{lem}
 For filtered graded $A$-modules $M, N$, we can calculate $\op{Ext}_-(M, N)$ by filtered graded projective resolutions of $M$. I.e, we have
\[
 \op{Ext}^i_{A-\op{gr}, -}(M,N) = H^i(\op{Hom}_{A\op{-gr},-}(P^\bullet, N))
\] 
, where $P^\bullet \rightarrow M \rightarrow 0$ is any filtered graded projective resolution of $M$.
\label{165514_26Jun21}
\end{lem}
\begin{proof}
 First, we construct a spectral sequence which is convergent to $H^i(\op{Hom}_{A-\op{gr}, -}(P^{\bullet}, N))$ by using the idea of the proof of \cite[Proposition 1.3]{drezet1985}. We have a natural filtration of a chain complex $C = \op{Hom}_{A-\op{gr},-}(P^\bullet, N)$ :
\[
\def\objectstyle{\scriptstyle}
\def\labelstyle{\scriptscriptstyle}
 \xymatrix@C=10pt{ C = F_0 C:& 0 \ar[r] & \op{Hom}_{A-\op{gr},-}(P^0, N) \ar[r]  & \op{Hom}_{A-\op{gr}, -}(P^1, N)  \ar[r]& \op{Hom}_{A-\op{gr}, -}(P^2, N)  \ar[r] & \cdots \\
F_1 C :& 0 \ar[r] & \op{Hom}_{A-\op{gr}, -1}(P^0, N) \ar[r] \ar@{^{(}->}[u] & \op{Hom}_{A-\op{gr}, -1}(P^1, N) \ar[r] \ar@{^{(}->}[u] & \op{Hom}_{A-\op{gr}, -1}(P^2, N) \ar[r] \ar@{^{(}->}[u] &\cdots \\
F_2 C :& 0 \ar[r] & \op{Hom}_{A-\op{gr}, -2}(P^0, N) \ar[r] \ar@{^{(}->}[u]& \op{Hom}_{A-\op{gr}, -2}(P^1, N) \ar[r] \ar@{^{(}->}[u]& \op{Hom}_{A-\op{gr}, -2}(P^2, N) \ar[r] \ar@{^{(}->}[u]& \cdots \\
&&  \vdots \ar@{^{(}->}[u] & \vdots \ar@{^{(}->}[u] & \vdots  \ar@{^{(}->}[u]
}
\]
,where $\op{Hom}_{A-\op{gr}, -i}(P^j, N) = \{f \in \op{Hom}_{A-\op{gr}}(P^j , N) \mid  f((P^j)_k) \subset N_{k-i} \} (i \in \mathbb{N})$. Although filtrations in Section \ref{153759_23Jun21} are ascending, note that the above filtration is descending. So, we have a spectral sequence by \cite[Theorem 5.5.1]{weibel1994introduction} as follows :
\[
 H^{p+q}(\op{Hom}_{A\op{-gr}, -}(P^\bullet, N)) \Leftarrow E^{pq}_1 := \begin{cases} H^{p+q}(F_p C / F_{p+1} C) & p \geq 0 \\
0 & p < 0
\end{cases}.
\]
Because $P^i$ is filtered graded projective $A$-module, by \cite[I, Lemma 6.4]{nastasescu2006graded}, we have $F_pC/F_{p+1}C $ is isomorphic to $\oplus_{i}\op{Hom}_{A-\op{gr}}(\op{gr}_i(P^\bullet), \op{gr}_{i-p}(N))$ if $p \geq 0$. So, we have 
\[
  H^{p+q}(F_p C/F_{p+1} C) = \oplus_i \op{Ext}^{p+q}_{A-\op{gr}}(\op{gr}_i (M), \op{gr}_{i-p} (N)) \quad \text{if } p \geq 0.
\]
We have a spectral sequence converging to $H^{p+q}(\op{Hom}_{A-\op{gr}, -}(P^\bullet, N))$ and $ \op{Ext}_{A-\op{gr}, -}(M,N)$ (Theorem \ref{161710_25Jun21}, Remark \ref{205937_25Jun21}). Thus, we have  $H^{i}(\op{Hom}_{A-\op{gr}, -}(P^\bullet, N)) \simeq \op{Ext}^i_{A-\op{gr}, -}(M,N)$.
\end{proof}

We can describe the tangent space and the obstruction space of filtered graded $A$-module $M$ from the above discussion and Lemma \ref{165514_26Jun21}.

\begin{prop}
\label{173918_1Jul21}
 For any filtered graded $A$-module $M$, then
\begin{align*}
&\text{The infinitesimal deforamtion of $M$ is classified by }  \op{Ext}^1_{A-\op{gr}, -}(M,M), \\
&\text{The obstructions of the deformation of $M$ are contained in } \op{Ext}^2_{A-\op{gr}, -}(M,M).
\end{align*} 
\end{prop}
\begin{proof}
This is obtained from Lemma \ref{165514_26Jun21} and the following isomorphisms: 
\begin{small}
 \begin{align*}
\op{Hom}_{k-\op{gr}}(A^{\otimes \bullet}, \op{Hom}_{k}(M, M)) &\simeq
 \op{Hom}_{A-A-\op{gr}}(B(A,A), \op{Hom}_{k,-}(M,M)) \\
&\simeq \op{Hom}_{A-k-\op{gr},-}(B(A,A) \otimes_A M, M)
\end{align*}
\end{small}
, where $B(A,A)$ is the Bar resilution of $A$.
Note that we have the above from \cite[Lemma 8.1]{nastasescu2006graded} and a property of Bar resolution and $B(A,A)\otimes_A M \rightarrow M \rightarrow 0$ is a filtered graded projective resolution of $M$
(for the non filtered graded case, see \cite[Lemma 9.1.9]{weibel1994introduction}).
\end{proof}

Next, we compare the deformation theories of filtered sheaves on $X$ and filtered graded $A$-modules.
We use the method of \cite{ciocan2001} where them of sheaves on $X$ and graded $A$-modules are compared.

 \begin{prop}
(\cite{walter1995components})
\label{173948_1Jul21}
  For any filtered coherent sheaf $\mathscr{F}$, we can describe the deformation theory of $F$ as follows:
 \begin{align*}
&\text{The infinitesimal deformation of $\mathscr{F}$ is classified by }\op{Ext}^1_{-}(\mathscr{F}, \mathscr{F}), \\
&\text{The obstructions of the deformation of $\mathscr{F}$ are contained} = \op{Ext}^2_{-}(\mathscr{F}, \mathscr{F}).
\end{align*}
 \end{prop}

\begin{lem}
\label{173516_29Sep21}
 \begin{itemize}
  \item[(a)] If $\mathscr{F}, \mathscr{G}$ are filtered coherent sheaves on $X$, then 
\[
  \op{Ext}^i_{-}(\mathscr{F}, \mathscr{G}) = \operatorname*{lim}_{\underset{p}{\rightarrow}} \op{Ext}^i_{A-\op{gr}, -}(\Gamma^{\op{fil}}_{\geq p}(\mathscr{F}), \Gamma^{\op{fil}}_{\geq p} (\mathscr{G})). 
\] 
\item[(b)] If $M, N$ are finitely generated filtered graded $A$-module and take a nonnegative integer $p$, then for any $q \gg 0$, we have 
\[
 \op{Ext}^i_{A-\op{gr},-}(M,N) = \op{Ext}^i_{A-\op{gr}, -}(M_{\leq q}, N_{\leq q}).
\]
	   Moreover, if let $Y$ be a projective scheme and let  $\mathscr{M}, \mathscr{N}$ be filtered graded $A \otimes \mathscr{O}_{Y}$-modules whose coponents are locally free sheaves on $Y$, then we can choose a $q$ such that for any pair $(\mathscr{M}_{y}, \mathscr{N}_{y}) (y \in Y)$, the above equality holds.
 \end{itemize}
\end{lem}

\begin{proof}
(a) Let $\mathcal{S}$ be a category whose objects are finitely generated $A$-modules and morphisms between two objects $M$ and $N$ are 
\[
 \op{Hom}_{\mathcal{S}}(M, N) := \op{lim}_{\underset{p}{\rightarrow}} \op{Hom}_{A-\op{gr}}(M_{\geq p}, N_{\geq p})
\]
(\cite[page.406-407]{ciocan2001}). Note that the category $\op{Coh}(X)$ is equivalent to the category $\mathcal{S}$ (see \cite[Theorem 1.2.2]{ciocan2001} or \cite{serre1955faisceaux}). 
If let $\op{Hom}_{\mathcal{S}, -}(M, N)$ be the subset of $\op{Hom}_{\mathcal{S}}(M, N)$ consisting of filtered preserving morphisms from  $M$ to  $N$.
Then, we have  $ \op{Hom}_{\mathcal{S}, -}(M, N) = \operatorname*{lim}_{\underset{p}{\rightarrow}}\op{Hom}_{A-\op{gr}, -}(M_{\geq p}, N_{\geq p})$.
Let $0 \rightarrow \Gamma^{\op{fil}}(G) \rightarrow I^{\bullet}$ be a filtered injective resolution of $\Gamma^{\op{fil}}(G)$ and let $\mathscr{I}^i :=\tilde{I^i}$.
The natural functor from the $A-$graded modules to $\mathcal{S}$ is an exact functor(\cite[Theorem 6.7]{smith-gradedrings}). 
So, a filtered graded resolution in the category of graded $A$-modules is one in $\mathcal{S}$.  
Any injective object in the category of the graded $A$-modules is injective in $\mathcal{S}$ from a direct calculation and the definition of injections in $\mathcal{S}$.
It follows that  $0 \rightarrow \Gamma^{\op{fil}}(G) \rightarrow I^{\bullet}$ is also a filtered injective resolution in $\mathcal{S}$.
So, $0 \rightarrow \mathscr{G} \rightarrow \mathscr{I}^{\bullet}$ is a filtered injective resolution of $\mathscr{G}$.
Thus, we have
\begin{align*}
 \op{Ext}^i_{-}(\mathscr{F}, \mathscr{G}) &= H^i(\op{Hom}_{-}(\mathscr{F}, \mathscr{I}^{\bullet})) = H^i(\op{Hom}_{\mathcal{S}, -}(\Gamma^{\op{fil}}(\mathscr{F}), I^{\bullet}) \\
&= H^i (\operatorname*{lim}_{\underset{p}{\rightarrow}} \op{Hom}_{A-\op{gr}, -}(\Gamma^{\op{fil}}_{\geq p}(\mathscr{F}), I^{\bullet}_{\geq p})) \\
&= \operatorname*{lim}_{\underset{p}{\rightarrow}} H^i ( \op{Hom}_{A-\op{gr}, -}(\Gamma^{\op{fil}}_{\geq p}(\mathscr{F}), I^{\bullet}_{\geq p})) 
\end{align*}
In order to show (a), we have to prove the following claim.
\begin{cl}
\label{181958_29Jun21}
 For $i,p \geq 0$,
 \[
  H^i ( \op{Hom}_{A-\op{gr}, -}(\Gamma^{\op{fil}}_{\geq p}(\mathscr{F}), I^{\bullet}_{\geq p})) 
=  \op{Ext}^i_{A-\op{gr}, -}(\Gamma^{\op{fil}}_{\geq p}(\mathscr{F}), \Gamma^{\op{fil}}_{\geq p}(\mathscr{G}))
\]

\end{cl}
\begin{proof}
[Proof of Claim \ref{181958_29Jun21}]
For simplicity, $M := \Gamma^{\op{fil}}(\mathscr{F})$.
 First, we have a spectral sequence in the same way as in the proof of Lemma \ref{165514_26Jun21} :
\[
  H^{i+j}(\op{Hom}_{A-\op{gr}, -}(M_{\geq p}, I^{\bullet}_{\geq p})) \Leftarrow E^{ij}_1 := \begin{cases} H^{i+j}(F_i C_{\geq p} / F_{i+1} C_{\geq p}) & i \geq 0 \\
0 & i < 0
\end{cases}
\]
, where $F_iC_{\geq p} := \{0 \rightarrow \op{Hom}_{A-\op{gr}, -i}(M_{\geq p}, I^0_{\geq p}) \rightarrow  \op{Hom}_{A-\op{gr}, -i}(M_{\geq p}, I^1_{\geq p} ) \rightarrow \cdots \}$.
Then, we see that
\begin{equation}
   F_i C_{\geq p} / F_{i+1} C_{\geq p} =  \oplus_k \op{Hom}_{A-\op{gr}}(\op{gr}_k (M_{\geq p}), \op{gr}_{k-i} (I^\bullet_{\geq p})) \quad  \text{if } i \geq 0. \label{162727_1Jul21}
\end{equation}
To prove \ref{162727_1Jul21}, we use the idea of the proof of \cite[Chapitre V, Lemme 1.4.2]{illusie1971}.
First, note that each $I^{i}_{\geq p}$ is a direct sum of truncations of graded injective modules because each $I^i$ is a direct sum of graded injective modules.
When $0 = N^0 \subset N^1 \subset \cdots \subset N^t =: N$ and $0 = J^0 \subset J^1 \subset \cdots \subset J^t := J$ are filtered graded $A$-modules and $J^{i+1}/J^i$ are injective objects (so each $J^\bullet$ is a direct sum of injective objects), then we have 
\begin{align*}
  \op{Hom}_{A-\op{gr}, -}(N_{\geq p}, J_{\geq p}) & = \op{Hom}_{A-\op{gr}, -}(N_{\geq p}, J)\\
 &=  \oplus_{k=0} \op{Hom}_{A-\op{gr}}((N/N^{t-k-1})_{\geq p}, J^{t-k}/J^{t-k-1})\\
& =  \oplus_{k=0} \op{Hom}_{A-\op{gr}}((N/N^{t-k-1})_{\geq p},(J^{t-k}/J^{t-k-1})_{\geq p}).
\end{align*}
\begin{rmk}
The second equal above can be obtained by the following correspondence. 
\[
\xymatrix@R=5pt@C=15pt{ \op{Hom}_{A-\op{gr}, -}(N_{\geq p}, J) \ar@<0.5ex>[r] &  \oplus_{k=0}^{t-1} \op{Hom}_{A-\op{gr}}((N/N^{t-k-1})_{\geq p}, J^{t-k}/J^{t-k-1}) \ar@<0.5ex>[l] \\
\psi  \ar@{}[r]^(.20){}="a"^(.70){}="b" \ar@{|->} "a";"b" &  \oplus_{k=0}^{t-1} (\op{pr}_{t-k} \circ \psi) \\
\oplus_{k=0}^{t-1}(\varphi_{t-k} \circ \pi_{t-k-1})  & \oplus_{k=0}^{t-1} \varphi_{t-k}  \ar@{}[l]^(.20){}="a"^(.70){}="b" \ar@{|->} "a";"b"
}\]
where $\op{pr}_{t-k} : J = \oplus_{k=0}^{t-1} J^{t-k}/J^{t-k-1} \rightarrow J^{t-k}/J^{t-k-1}$ and $\pi_{t-k}: N \rightarrow N/N^{t-k}$ are the natural projections.
\end{rmk}
So, if let $M := (M = M^s \supset M^{s-1} \supset \cdots \supset M^{0} = 0) $ as a filtered module, we have
\begin{small}
\begin{align*}
 F_i C_{\geq p} &= \op{Hom}_{A-\op{gr}, -i}(M_{\geq p}, I^\bullet_{\geq p})
= \oplus_{k=0} \op{Hom}_{A-\op{gr}}(M_{\geq p}/(M^{s-i-k-1})_{\geq p}, \op{gr}_{s-i-k}(I^\bullet_{\geq p})), \\
   F_i C_{\geq p} / F_{i+1} C_{\geq p} &= 
 \bigoplus_{k=0} \frac{\op{Hom}_{A-\op{gr}}((M/M^{ s-k-1})_{\geq p}, \op{gr}_{ s-i-k}(I^\bullet_{\geq p}))}{\op{Hom}_{A-\op{gr}}((M/M^{ s-k})_{\geq p}, \op{gr}^{ s-i-k}(I^\bullet_{\geq p}))}.
\end{align*} 
\end{small}
Then, we apply the functor $\op{Hom}_{A-\op{gr}}(-, \op{gr}_{s-i-k}(I^\bullet_{\geq p}))$ to the exact sequence
\[
 0 \rightarrow \op{gr}_{s-k}(M_{\geq p}) \rightarrow (M/M^{s-k-1})_{\geq p} \rightarrow (M/M^{s-k})_{\geq p} \rightarrow 0 .
\]
Because $I^\bullet_{s-i-k}$ is an injective module and $\op{Hom}_{A-\op{gr}}(L_{\geq p}, \op{gr}_{s-i-k}(I^\bullet_{\geq p})) =$ $ \op{Hom}_{A-\op{gr}}(L_{\geq p}, \op{gr}_{s-i-k}(I^\bullet_{\geq p}))$ for a graded $A$-module $L$, we have an exact sequence
\begin{align*}
  0 &\rightarrow \op{Hom}_{A-\op{gr}}((M/M^{s-k})_{\geq p}, \op{gr}_{s-i-k}(I^\bullet_{\geq p})) 
\rightarrow \op{Hom}_{A-\op{gr}}((M/M^{s-k-1})_{\geq p}, \op{gr}_{s-i-k}(I^\bullet_{\geq p}))\\
 &\rightarrow \op{Hom}_{A-\op{gr}}(\op{gr}_{s-k}(M_{\geq p}), \op{gr}_{s-i-k}(I^\bullet_{\geq p})) \rightarrow 0.
\end{align*}
Therefore, we have
\[
  F_i C_{\geq p} / F_{i+1} C_{\geq p} =  \oplus_{k=0} \op{Hom}_{\op{gr}}(\op{gr}_{s-k}(M_{\geq p}), \op{gr}_{s-i-k}(I^\bullet_{\geq p})).
\]
This is the desired equality.

Finally, because $\op{Ext}^i_{A-\op{gr}}(\op{gr}_{s-k}(M_{\geq p}), \op{gr}_{s-i-k}(I^\bullet_{\geq p})) = 0$ (\cite[Lemma 4.3.7]{ciocan2001}), we have $\op{gr}_{s-i-k}(I^\bullet_{\geq p})$ is $\op{Hom}_{A-\op{gr}}((M^{s-k}/M^{s-k-1})_{\geq p}, -) -$acyclic. 
Thus, we have
\[
 H^{i+j}(F_i C_{\geq p} / F_{i+1} C_{\geq p}) = \oplus \op{Ext}^{i+j}_{A-\op{gr}}(\op{gr}_{k}(M_{\geq p}), \op{gr}_{k-i}(N_{\geq p})).
\]
So, in the same way as in the proof of Lemma \ref{165514_26Jun21}, we have the claim.
\end{proof}

Therefore, from Claim \ref{181958_29Jun21} and the above discussion, we have 
\[
 \op{Ext}^i_{-}(\mathscr{F}, \mathscr{G}) = \operatorname*{lim}_{\underset{p}{\rightarrow}} \op{Ext}^i_{A-\op{gr}, -}(\Gamma^{\op{fil}}_{\geq p}(\mathscr{F}), \Gamma^{\op{fil}}_{\geq p}(\mathscr{G})).
\]

(b) Next, it is sufficient to prove that if let $F^\bullet \rightarrow M \rightarrow 0$
 be a filtered graded free resolution of $M$ 
, then for $q \gg 0$
\[
 \op{Ext}^i_{A-\op{gr}, -}(M_{\leq q}, N_{\leq q}) = H^i(\op{Hom}_{A-\op{gr},-}(F^{\bullet}_{\leq q}, N_{\leq q})).
\]
First note that the truncated complex  $F^\bullet_{\leq q} \rightarrow M_{\leq q} \rightarrow 0$ is a respolution of $M_{\leq q}$ (which is not free). 
In the same way as in the proof of \ref{165514_26Jun21}, we  can get a spectral sequence 
\[
 H^{i+j}(\op{Hom}_{A-\op{gr}, -}(P^\bullet_{\leq q}, N_{\leq q})) \Leftarrow E^{ij}_1 := \begin{cases} H^{i+j}(F_i C_{\leq q} / F_{i+1} C_{\leq q}) & i \geq 0 \\
0 & i < 0
\end{cases}
\]
, where $F_i C_{\leq q} := \{ 0  \rightarrow \op{Hom}_{A-\op{gr}, -i}(P^0_{\leq q}, N_{\leq q}) \rightarrow  \op{Hom}_{A-\op{gr}, -i}(P^1_{\leq q}, N_{\leq q}) \rightarrow \cdots \}$.
And, we have 
\[
 H^{i+j}(F_i C_{\leq q} / F_{i+1} C_{\leq q}) = \oplus_k \op{Ext}^{i+j}_{A-\op{gr}}(\op{gr}_{k}(M_{\leq q}), \op{gr}_{k-i}(N_{\leq q})) \quad \text{if } i \geq 0
\]
because the similar equality to \ref{162727_1Jul21} holds and $\op{Ext}^i_{A-\op{gr}}(M_{\leq q}, N_{\leq q})$ are calculated by using the truncated resolution $F^\bullet_{\leq q} \rightarrow M_{\leq q} \rightarrow 0$ (\cite[page 437]{ciocan2001}). 

The latter claim is proved samely by using the fact any coherent sheaf on a projective scheme $Y$ has a resolution by graded $A \otimes \mathscr{O}_Y$-modules $\mathscr{F}^\bullet$ which are of the form $\mathscr{F}^i = A \otimes \mathscr{E}^\bullet$ such that each $\mathscr{E}^i$ is locally free sheaves of finite rank (see also \cite[page 437]{ciocan2001}).
\end{proof}

\begin{cor}
For any $i$, there exists $q\gg p \gg 0$ such that
\[
 \op{Ext}^i_{-}(\mathscr{F}, \mathscr{F}) = \op{Ext}^i_{A-\op{gr}, -}(\Gamma^{\op{fil}}_{[p,q]}(\mathscr{F}), \Gamma^{\op{fil}}_{[p,q]}(\mathscr{F}))
\]
for all $k$-points $\mathscr{F}$ of $\mathcal{FC}\op{oh}^{\text{HN}}_{(\alpha_1, \cdots, \alpha_s)}(X)$.
 \label{160735_17Jan22}
\end{cor}
\begin{proof}
 From (a) of Lemma we have \ref{173516_29Sep21}, $\op{Ext}^i_{-}(\mathscr{F},\mathscr{F}) = \op{Ext}^i_{A-\op{gr,-}}(\Gamma^{\op{fil}}_{\geq p}(\mathscr{F}),\Gamma^{\op{fil}}_{\geq p}(\mathscr{F}))$ for $0 \ll p$. 
Moreover, we assume that $p,q$ are independent of $\mathscr{F}$ because $\mathcal{FC}\op{oh}^{\text{HN}}_{(\alpha_1, \cdots, \alpha_s)}(X)$ is of finite type.
Then, we apply (b) of the lemma.
\end{proof}

\begin{prop}
 $\Gamma^{\op{fil}}_{[p,q]}$ is an \'etale morphism if $q \gg  p \gg  0$
\label{171246_1Jul21}
\end{prop}

\begin{proof}
 This is followed from Proposition \ref{173918_1Jul21}, Proposition \ref{173948_1Jul21} and Corollary \ref{171246_1Jul21}.
For the non-filtered case, see \cite[Proposition 3.3]{behrend2014}.
\end{proof}

Finally, we can get the main result in this subsection by combining Lemma \ref{224726_27Jun21} and Proposition \ref{171246_1Jul21}.
\begin{cor}
  $\Gamma^{\op{fil}}_{[p,q]}$ is an open immersion if $q \gg  p \gg  0$
\label{215140_1Jul21}
\end{cor}

\subsection{Explicit descriptions}
In this subsection, we describe the image of $\Gamma^{\op{fil}}_{[p,q]}|_{ \mathcal{FC}\op{oh}^{\op{HN}}_{(\alpha_1, \cdots, \alpha_s)}(X)}$.
\begin{thm}
 For $q \gg p' \gg p \gg 0$, the functor $\Gamma^{\op{fil}}_{[p,q]}$ induces an open immersion
\[
 \Gamma^{\op{fil}}_{[p,q]}|_{\mathcal{FC}\op{oh}^{\op{HN}}_{(\alpha_1, \cdots, \alpha_s)}(X)} : \mathcal{FC}\op{oh}^{\op{HN}}_{(\alpha_1, \cdots, \alpha_s)}(X) \rightarrow \mathcal{FM}\op{od}^{[p,q], \op{sfg}}_{(\alpha_1, \cdots , \alpha_s)}(A)
\]
whose image is equal to the locus of whose truncations into the intervals $[p,q]$ and $[p',q]$ are HN-filtrations respectively.
We denote this open substack of $\mathcal{FM}\op{od}^{[p,q], \op{sfg}}_{(\alpha_1, \cdots , \alpha_s)}(A)$ by $\mathcal{FM}\op{od}^{[p,q]_{\op{HN}}, \op{sfg}, [p',q]_{\op{HN}}}_{(\alpha_1, \cdots , \alpha_s)}(A)$.
\label{150320_6Jul21}
\end{thm}
\begin{proof}
\item
\subsubsection*{Step 1}
For $p \gg 0$ and any sheaf $\mathscr{F}$ with HN-filtration $0 =\mathscr{F}^0 \subset \mathscr{F}^1 \subset \cdots \subset \mathscr{F}^s = \mathscr{F} $, all $\Gamma^{\op{fil}}_{[p,q]}(\mathscr{F}^{i+1}/\mathscr{F}^i)$ are generated in degree $p$ because $\mathscr{F}^{i+1}/\mathscr{F}^{i}$ is assumed to be $p$-regular.
It follows that $\Gamma^{\op{fil}}_{[p,q]}|_{ \mathcal{FC}\op{oh}^{\op{HN}}_{(\alpha_1, \cdots, \alpha_s)}(X)}$ is an open immersion from Corollary \ref{215140_1Jul21} for $q \gg p \gg 0$.

 \subsubsection*{Step 2}

Next we show any HN-filtration in $\mathcal{FC}\op{oh}_{(\alpha_1, \cdots, \alpha_s)}(X)$ is sent to a HN-filtration in $\mathcal{FM}\op{od}^{[p,q], \op{sfg}}_{(\alpha_1, \cdots , \alpha_s)}(A)$ for $q \gg p \gg 0$.
Note that for $q \gg p\gg 0$, any coherent sheaf $\mathscr{G}$ on $X$ whose Hilbert polynomial is one of $\{\alpha_1,\cdots, \alpha_s\}$ is semistable if and only if $\Gamma_{[p,q]}(\mathscr{G})$ is semistable with respect to $(\alpha(q), -\alpha(p))$-stability by \cite[Theorem 3.7]{behrend2014}.
For two sheaves $\mathscr{G}_1,\mathscr{G}_2$, $\mu_{(\alpha(q), -\alpha(p))}(\Gamma_{[p,q]}\mathscr{G}_1) > \mu_{(\alpha(q), -\alpha(p))}(\Gamma_{[p,q]}\mathscr{G}_2)$ if and only if $h^0(\mathscr{G}_1(p))/h^0(\mathscr{G}_1(q)) > h^0(\mathscr{G}_2(p))/h^0(\mathscr{G}_2(q))$ from the following calculations (a similar statement is mentioned in \cite[Lemma 5.2]{hoskins2018stratifications}, where representations of quivers are treated and the slope treated there is slightly different from ours):
{\small
\begin{align*}
 \mu_{(\alpha(q), -\alpha(p))}(\Gamma_{[p,q]}(\mathscr{G}_i)) &= \frac{\alpha(q)h^0(\mathscr{G}_i(p))-\alpha(p)h^0(\mathscr{G}_i(q))}{h^0(\mathscr{G}_i(p))+h^0(\mathscr{G}_i(q))} 
= -\alpha(p)+\frac{(\alpha(p)+\alpha(q))h^0(\mathscr{G}_i(p)}{h^0(\mathscr{G}_i(p))+h^0(\mathscr{G}_i(q))} \quad (i= 1,2),
\end{align*}
\begin{align*}
&\mu_{(\alpha(q), -\alpha(p))}(\Gamma_{[p,q]}(\mathscr{G}_1))- \mu_{(\alpha(q), -\alpha(p))}(\Gamma_{[p,q]}(\mathscr{G}_2)) \\
&= (\alpha(p)+\alpha(q))\left(\frac{h^0(\mathscr{G}_1(p))}{h^0(\mathscr{G}_1(p))+h^0(\mathscr{G}_1(q))} - \frac{h^0(\mathscr{G}_2(p))}{h^0(\mathscr{G}_2(p))+h^0(\mathscr{G}_2(q))} \right) \\
&= \frac{(\alpha(p)+\alpha(q))h^0(\mathscr{G}_1(q))h^0(\mathscr{G}_1(q))}{(h^0(\mathscr{G}_1(p))+h^0(\mathscr{G}_1(q)))(h^0(\mathscr{G}_2(p))+h^0(\mathscr{G}_2(q)))} \left(\frac{h^0(\mathscr{G}_1(p))}{h^0(\mathscr{G}_1)(q)}-\frac{h^0(\mathscr{G}_2(p))}{h^0(\mathscr{G}_2)(q)} \right).
\end{align*}
}
If let $\alpha(t) := \sum^s_{i=1} \alpha_i(t)$, this proves that $\Gamma_{[p,q]}(\mathscr{F}^{i+1}/\mathscr{F}^{i})$ is semistable with respect to $(\alpha(q), -\alpha(p) )$ and $\mu_{(\alpha(q), -\alpha(p) )}(\Gamma_{[p,q]}(\mathscr{F}^{i+1}/\mathscr{F}^{i})) > \mu_{(\alpha(q), -\alpha(p) )}(\Gamma_{[p,q]}(\mathscr{F}^{i+2}/\mathscr{F}^{i+1}))$ (cf.\cite[Theorem 5.7]{hoskins2018stratifications}).
In the same way, we can show that for $q \gg p' \gg p \gg 0$ and any HN-filtration $\mathscr{F} \in \mathcal{FC}\op{oh}_{(\alpha_1,\cdots, \alpha_s)}(X)$, $\Gamma_{[p,q]}(\mathscr{F})|_{\geq p'}$ is also a HN-filtration with respect to $(\alpha(q), -\alpha(p'))$-stability.

\vskip.5\baselineskip

In the rest of the proof, we show the following :
 For $q \gg p' \gg p \gg 0$, any $M \in \mathcal{FM}\op{od}^{[p,q], \op{sfg}}_{(\alpha_1, \cdots , \alpha_s)}(A)$ such that $M_{\geq p'}$ is a HN-filtration is sent to a object of $\mathcal{FC}\op{oh}^{\op{HN}}_{(\alpha_1, \cdots, \alpha_s)}(X)$.

\subsubsection*{Step 3}

Let $M= M^s \supset \cdots \supset M^1 \supset M^0 =0$ be a filtered graded $A$-module in $[p,q]$ of dimension $\alpha|_{[p,q]}$ and generated in degree $p$. Then, we have exact sequences
\[
  0 \rightarrow K^i \rightarrow A_{[0, q-p]} \otimes_{k} M^i_{p} \rightarrow M^{i} \rightarrow 0
\]
,where $K^i := \op{Ker}(A_{[0, q-p]} \otimes_{k} M^{i}_p \rightarrow M^{i})$.
They are compatible with the filtrations of $M^\bullet, A_{[0, q-p]} \otimes_{k} M^{\bullet }_p $ and $K^{\bullet}$ when we think of them as filtered modules.
Then, if let $K'^i \subset A\otimes M^{i}_p$ be a submodule generated by $K_i$ in $ A\otimes M^{i}_p$, then $\mathscr{S}(M^i) \simeq \widetilde{A \otimes_k M^{i}_p/ {K'}^{i}}$ where $\widetilde{A \otimes_k M^{i}_p/ {K'}^{i}}$ is the associated sheaf of $A \otimes_k M^{i}_p/ K'^{i}$ (see \cite{serre1955faisceaux}, \cite[Theorem 3.10]{behrend2014}).
Note that $K'^{i}$ is generated in degree $p+1$ if $p \gg 0$.
The functor $\; \widetilde{} \; $ is exact.
 So, any filtered graded $A$-module in $ \mathcal{FM}\op{od}^{[p,q], \op{sfg}}_{(\alpha_1, \cdots , \alpha_s)}(A)$ is sent a filtered coherent sheaf on $X$ by $\mathscr{S}$.
 

In the proof of \cite[Theorem 3.10]{behrend2014}, any finitely $[p,q]$-graded $A$-module $M$ in degree $p$ with dimension vector $\alpha|_{[p,q]}$ is sent to a $p'$-regular coherent sheaf with Hilbert polynomial $\alpha$ and $\Gamma_{[p,q]}(\mathscr{S}(M))_{\geq p'} \simeq M_{\geq p'}$ for $q > p' \gg p \gg 0$.
The choice of $p'$ is dependent only on $p$.
So, a similar claim holds :
for any $(0 = M^0 \subset M^1 \subset \cdots \subset M^s =M) \in \mathcal{FM}\op{od}^{[p,q], \op{sfg}}_{(\alpha_1, \cdots , \alpha_s)}(A)$, it holds that each $\mathscr{S}(M_{i+1})/\mathscr{S}(M_i) = \mathscr{S}(M^{i+1}/M^i)$ is $p'$-regular with Hilbert polynomial $\alpha_{i+1}$ and $\Gamma_{[p,q]}(\mathscr{S}(M^i))_{\geq p'} \simeq (M^i)|_{\geq p'}$ for $q \gg p' \gg p \gg 0$.

 \subsubsection*{Step 4}


We take $q \gg p' \gg p \gg 0$ so that we have
\begin{inparaenum}
 \item : the last sentence of Step $4$ holds and
\item : any coherent sheaf $\mathscr{G}$ on $X$ with Hilbert Polynomial $\alpha_i$ is semistable if and only if $\Gamma_{[p,q]}(\mathscr{G})$ (resp. $\Gamma_{[p',q]}(\mathscr{G})$) is semistable with respect to $(\alpha_i(q), \alpha_i(p))$-stability (resp. $(\alpha_i(q), \alpha_i(p'))$-stability) for all $i$ (cf. Step 2).
\end{inparaenum}
Then, any $M \in \mathcal{FM}\op{od}^{[p,q], \op{sfg}}_{(\alpha_1, \cdots , \alpha_s)}(A)$ is sent to a filtered sheaf in $\mathcal{FC}\op{oh}_{(\alpha_1,\cdots, \alpha_s)}(X)$ such that the $\mathscr{S}(M^{i+1}/M^i)$ are semistable because $(\alpha(q), -\alpha(p') )$-stability is equivalent to $(\alpha_{i+1}(q), -\alpha_{i+1}(p'))$-stability.
So, $\mathscr{S}(M)$ is a HN-filtration and $p$-regular.

\subsubsection*{Step 5}

 We show $\Gamma_{[p,q]}(\mathscr{S}(M)) \simeq M$.
Note that each $(\mathscr{S}(M^{i+1}/M^i)_p, \mathscr{S}(M^{i+1}/M^i)_q)$ is a semistable Kronecker module.
Then, it follows that $\Gamma(\mathscr{S}(M^{i+1}/M^i)(q)) \simeq (M^{i+1}/M^i)_q$ from the choice of $q$.
We also have the natural morphism $\Gamma(\mathscr{S}(M^{i+1}/M^i)(p)) \rightarrow (M^{i+1}/M^i)_p$ is injective by the above note (see also \cite[Prop 5.13]{alvarez2007functorial}).
Moreover this is an isomorphism because of the dimensions of $\Gamma(\mathscr{S}(M^{i+1}/M^i)(p)), (M^{i+1}/M^i)_p$ (we use the $p$-regurality of $\mathscr{S}(M^{i+1}/M^i)$ ).
So, we have a commutative diagram
{\small
\[\xymatrix{
A_{[0, q-p]} \otimes_k (M^{i+1}/M^i)_p \ar[r] \ar[d]^{\simeq} & M^{i+1}/M^i \ar[d]\\
A_{[0, q-p]} \otimes_k \Gamma(\mathscr{S}(M^{i+1}/M^i)(p)) \ar@{->>}[r] & \Gamma_{[p,q]}\mathscr{S}(M^{i+1}/M^i) \ar@{}[lu]|{\circlearrowright}. 
}\]}
It follows that the right vertical arrow is isomorphism from the equality of the dimension vectors of $M^{i+1}/M^i$ and $\Gamma_{[p,q]}\mathscr{S}(M^{i+1}/M^i)$.
Here, we have 
{\small
\[\xymatrix{
0 \ar[r] & M^{i} \ar[r] \ar[d]& M^{i+1} \ar[d] \ar[r] & M^{i+1}/M^i \ar[r] \ar[d] & 0 \quad \text{(exact)}\\
0 \ar[r] & \Gamma_{[p,q]}\mathscr{S}(M^i) \ar[r] & \Gamma_{[p,q]}\mathscr{S}(M^{i+1}) \ar[r] \ar@{}[lu]|{\circlearrowright} & \Gamma_{[p,q]}\mathscr{S}(M^{i+1}/M^i) \ar[r] \ar@{}[lu]|{\circlearrowright} & 0 \quad \text{(exact)}
}\]}
, where the surjectivity of the arrow $\Gamma_{[p,q]}\mathscr{S}(M^{i+1}) \rightarrow \Gamma_{[p,q]}\mathscr{S}(M^{i+1}/M^i)$ comes from the $p$-regularity of $\mathscr{S}(M^i)$.
If the left and the right vertical arrows are isomorphism, then so is the middle one from the five lemma.
Therefore, inductively we can show that $M^i \simeq \Gamma_{[p,q]}\mathscr{S}(M^{i+1}/M^i)$ and this completes the proof.

\end{proof}


\subsection{Derived enhancement}
From now, we assume $X$ is smooth.	
Finally, we can define derived moduli stacks of Harder-Narasimhan filtrations.
This is our aim in this paper.

 Let $\alpha_1(t), \alpha_2(t), \cdots, \alpha_s(t) \in \mathbb{Q}[t]$ such that $\alpha_1(t) \succ \alpha_2(t) \succ \cdots \succ \alpha_s(t)$.
 We take integers $p, p', q$ so that Theorem \ref{150320_6Jul21} holds for $ \Gamma^{\op{fil}}_{[p,q]}|_{\mathcal{FC}\op{oh}^{\op{HN}}_{(\alpha_1, \cdots, \alpha_s)}(X)} : \mathcal{FC}\op{oh}^{\op{HN}}_{(\alpha_1, \cdots, \alpha_s)}(X) \rightarrow \mathcal{FM}\op{od}^{[p,q], \op{sfg}}_{(\alpha_1, \cdots , \alpha_s)}(A)$.
Then, we have the following diagram:
{\small
\begin{equation*}
\xymatrix@C=12pt{
& [\op{MC}(L_{[p,q],-})/P_{[p,q]}] \ar[d]^{\simeq} & \op{MC}(L_{[p,q],-}) \ar[l] \\ 
& \mathcal{FC}\op{oh}^{[p,q]}_{(\alpha_1, \cdots \alpha_s)}(A) 
& \mathcal{FC}\op{oh}^{[p,q] \op{sfg}}_{(\alpha_1, \cdots \alpha_s)}(A) \ar@{_{(}->}[l] 
& \mathcal{FM}\op{od}^{[p,q]_{\op{HN}}, \op{sfg}, [p',q]_\op{HN}}_{(\alpha_1, \cdots , \alpha_s)}(A) \simeq \mathcal{FC}\op{oh}^{\op{HN}}_{(\alpha_1, \cdots, \alpha_s)}(X). \ar@{_{(}->}[l]
}\label{225123_28Jan22}
\end{equation*}}
Here, we define a $P_{[p,q]}$-equivariant open subscheme $U_{[p,q]}$ of $L^1_{[p,q],-}$ to be 
{\small
\[
 U_{[p,q]} : = \left\{ {\begin{gathered} \lambda: A \otimes_k V_{[p,q]} \rightarrow V_{[p,q]} \in L^1_{[p,q],-} \text{ such that the truncations of } \lambda \text{ to } [p,q] \text{ and } [p',q] \text{ is } \\ \text{HN-filtrations respectively and } V_{[p,q]} \text{ is strongly finitely generated} \end{gathered}} \right\}.
\]}
Then, we have an isomorphism $[(U_{[p,q]} \cap \op{MC}(L_{[p,q],-}))/P_{[p,q]}] \simeq \mathcal{FM}\op{od}^{[p,q]_{\op{HN}}, \op{sfg}, [p',q]_\op{HN}}_{(\alpha_1, \cdots , \alpha_s)}(A)$.

On the other hand, We define a dg Lie algebra $ L_{\geq p,-}, L_-$ as follows:
\[
 L_{\geq p,-}:= \lim_{\underset{q}{\longleftarrow}} L_{[p,q],-} ,\quad  L_- := \lim_{\underset{p}{\longrightarrow}} L_{\geq p,-}.
\]
We also define $P_{\geq p},P$ to be 
\[
 P_{\geq p} = \op{Gr}_{\op{gr}}(V_{\geq p}) , \quad P = \lim_{\underset{p}{\longrightarrow}} P_{\geq p} .
\]

Let $\op{pr}_p^q: L_{\geq p,-} \rightarrow L_{[p,q],-}$ be the natural projections $(p,q \in \mathbb{Z})$ and let $\iota_p: L_{\geq p} \rightarrow L_{-} (p \in \mathbb{Z})$ be also the natural maps.
Then, we define $U_{\geq p}:= \bigcup_{q \gg p} (\op{pr}_p^q)^{-1}(U_{[p,q]})$ and $U := \bigcup_{p \gg 0} \iota_p(U_{\geq p}) $.
Note that $U_{\geq p} \subset L_{\geq p,-}$ and $U \subset L_{-}$ are open immersions.
In addition, $U$ is $P$-equivariant.

\begin{dfn}
\label{204924_6Jan22}
 We define the derived moduli stack $\mathcal{RFC}\op{oh}^{\op{HN}}_{(\alpha_1, \cdots, \alpha_s)}(X)$ of Harder-Narasimhan filtrations of type $(\alpha_1, \cdots, \alpha_s)$ on $X$ is the restriction of the derived moduli stacks $[L_{-}/P]$ to $[U/P]$.
\end{dfn}

The derived moduli stacks $\mathcal{RFC}\op{oh}^{\op{HN}}_{(\alpha_1, \cdots, \alpha_s)}(X)$ is suiteble for the definition of derived moduli stacks of HN-filtrations.
They have the following property from the discussion from the discussion so far.

\begin{thm}
\label{160827_17Jan22}
We have the following.
\begin{itemize}
 \item[(a)] Let $\pi_0(\mathcal{RFC}\op{oh}^{\op{HN}}_{(\alpha_1, \cdots, \alpha_s)}(X)):= \op{Spec} (H^0(\mathscr{O}_{\mathcal{RFC}\op{oh}^{\op{HN}}_{(\alpha_1, \cdots, \alpha_s)}(X)} ))$. Then, 
\begin{equation}
 \pi_0\mathcal{RFC}\op{oh}^{\op{HN}}_{(\alpha_1, \cdots, \alpha_s)}(X) \simeq \mathcal{FC}\op{oh}^{\op{HN}}_{(\alpha_1, \cdots, \alpha_s)}(X).
\end{equation}
\item[(b)] If $0 = \mathscr{F}^0 \subset \mathscr{F}^1 \subset \cdots \subset \mathscr{F}^s = \mathscr{F}$ is a HN-filtration of type $(\alpha_1, \cdots, \alpha_s)$ on $X$, then 
\begin{equation}
  H^iT^{\bullet}_{[\mathscr{F}]}\mathcal{RFC}\op{oh}^{\op{HN}}_{(\alpha_1, \cdots, \alpha_s)}(X) \simeq \op{Ext}^i_{-}(\mathscr{F}, \mathscr{F}) \quad i \geq 0
\end{equation}
, where  $T^{\bullet}_{[\mathscr{F}]}\mathcal{RFC}\op{oh}^{\op{HN}}_{(\alpha_1, \cdots, \alpha_s)}(X)$ is the tangent dg-space of $\mathcal{RFC}\op{oh}^{\op{HN}}_{(\alpha_1, \cdots, \alpha_s)}(X)$ at $\mathscr{F}$.
\end{itemize}
\end{thm}

\begin{rmk}

\begin{itemize}
\item For the definition of tangent dg-spaces, see \cite[Section 1D]{behrend2014}, \cite[Def 5.1.5]{ciocan2002derived}.

\item Note that the derived stack $\mathcal{RFC}\op{oh}^{\op{HN}}_{(\alpha_1, \cdots, \alpha_s)}(X)$ is of infinite type. 
Actually, $U$ is not finite dimensional in general. 
But, $[(U \cap \op{MC}(L_-))/P]$ is finite dimensional because of the choice of $p,q$ and Theorem \ref{150320_6Jul21}, which is isomorphic to $\mathcal{FC}\op{oh}^{\op{HN}}_{(\alpha_1, \cdots, \alpha_s)}(X)$ in (a).
\item The constructions of derived moduli spaces in \cite{ciocan2001}, \cite{behrend2014} do not give the collect moduli spaces because higher cohomology of tangent complexes at some points in those spaces have discrepancies.
Our construction fixes this problem.

\end{itemize}
\end{rmk}

\section{Comparing with another construction of derived moduli stacks of filtered sheaves}

\begin{nota}
 \begin{itemize}
  \item $\op{SSet}$: the category of simplicial sets
\item $\op{Cat}$: the category of categories
\item $\op{Grpd}$: the category of groupoids
\item $\op{sCat}$ : the category of simplicial categories
\item $ \op{dg}_{\leq 0}\op{Alg}_k$ : the category of non-positive graded dg $k$-algabras
\item $\op{dg}_{\leq 0}\op{Nil}_{k}^{\flat}$: the full subcategory of $ \op{dg}_{\leq 0}\op{Alg}_k$ consisting of $C$ whose natural augmented map $C \rightarrow H^0(C)$ has a nilpotent kernel and $C_i = 0$ for all $i \ll 0$.
\item for $C \in \op{dg}_{\leq 0}\op{Alg}_k$, $\op{Spec}(C)$ is the associated affine dg-scheme (cf. \cite[Section 2.2]{ciocan2001}) 
 \end{itemize}
\end{nota}

In this section, we compare our construction of derived stacks of HN-filtrations with that of derived moduli stacks of filtered sheaves by Di Natale (\cite{di2017derived}).

\subsection{Translation of our construction}
\label{164506_16Jan22}

Since we construct derived stacks of HN-filtrations as dg-stacks, we translate them into derived stacks in the sense of To\"{e}n (\cite{toen2014derived}, \cite{toen2008homotopical}).
In this section, `` derived stacks and schemes'' means
derived stacks and schemes in the sense of To\"{e}n respectively. 
So, derived stacks and schemes are distinguished from dg-stacks and schemes respectively.

Let $L_-, P$ be as in Definition \ref{204924_6Jan22}.
Let $X$ be $\op{Spec}(L_-^1)$ and $\mathscr{R}_-$ be the sheaf of dg-algebras on $X$ which is associated to $L_-$.
In addition, $B_- := \Gamma(X, \mathscr{R}_-)$.
In the context of \cite[Section 3.3]{toen2004hag}, our construction of dg stacks of filtered modules is translated into the language of derived stack.
\begin{dfn}
 Let $\mathbb{R}\op{Spec}(B_-)$ be the derived scheme associated to $B_-$.
Then, from the action of $P$ on $B_-$, we define a simplicial derived affine scheme 
\[
\mathbb{R}\op{Spec}(B_-) \times_k^h P^\bullet : [n] \mapsto \mathbb{R}\op{Spec}(B_-) \times_k^h \overbrace{P \times_k^h \cdots \times_k^h P}^{(n-1) \text{-times}}.
\]
And, we define
\[
 [\mathbb{R}\op{Spec}(B_-) / P] := \op{Hocolim}_{n \in \Delta^{\text{op}}} ( \mathbb{R}\op{Spec}(B_-) \times_k^h P^n).
\]
\end{dfn}

\begin{dfn}
We set $\Omega_n$ to be the dg algebra 
\[
 k[t_0, t_1, \cdots, t_n, dt_0, dt_1, \cdots, dt_n]/(\sum t_i-1, \sum dt_i)
\]
, where $t_i$ are degree $0$ and $dt_i$ are degree $1$.
 
\end{dfn}

\vskip 1 \baselineskip

By using techniques in \cite{pridham2013constructing}, we consider a groupoid-valued functor.
 We define the groupoid-valued functor $\mathcal{F}_- : \op{dg}_{\leq 0}\op{Nil}_{k}^{\flat} \rightarrow \op{Grpd}$ as the stackification of the groupoid presheaf
\[
 C \mapsto [\op{MC}(L_- \otimes_k C)/ P \otimes_k C^0]_{\op{act}}
\]
in the strict \'etale topology (\cite[Def 2.17]{pridham2013constructing}), where $[\op{MC}(L_- \otimes_k C)/ P \otimes_k C^0]_{\op{act}}$ means the Deligne groupoid obtained from the action $P \otimes_k C_0 \curvearrowright \op{MC}(L_- \otimes_k C)$.

We define a simplicial enrichment $\underline{\mathcal{F}_-} : \op{dg}_{\leq 0}\op{Nil}_{k}^{\flat} \rightarrow [\Delta^{\op{op}}, \op{Cat}]$ as 
\[
 \underline{\mathcal{F}_-}(C) : [n] \mapsto \mathcal{F}_-(\tau_{\leq 0} ( C \otimes_k \Omega_n)) 
\]
, where $\tau$ means the canonical truncation functor of cochain complexes.
In addition, we applying the nerve functor $N$ and the right adjoint $\bar{W}$ to Illusie's total decalage functor (\cite[Def 1.25]{pridham2013constructing}) to $\underline{\mathcal{F}_-}$ and we get 
\[
 (\bar{W} \circ N) (\underline{\mathcal{F}_-}) :  \op{dg}_{\leq 0}\op{Nil}_{k}^{\flat} \rightarrow \op{SSet}.
\]

The following proposition is obtained as in the same way of \cite[Prop 3.16, Rmk 3.17]{pridham2013constructing}
\begin{prop}
 $(\bar{W} \circ N) (\underline{\mathcal{F}_-})$ is an almost finitely presented derived geometric $1$-stack.
Moreover the stacks $(\bar{W} \circ N) (\underline{\mathcal{F}_-}) $ and $ [\mathbb{R}\op{Spec}(B_-) / P]$ are weakly equivalent. 
\label{163437_16Jan22}
\end{prop}

\begin{dfn}
Let $\alpha_1, \alpha_2,\cdots, \alpha_s \in \mathbb{Q}[t]$ such that $\alpha_1 \succ \alpha_2 \succ \cdots \succ \alpha_s$. 
 Let $\tilde{U} \subset \mathbb{R}\op{Spec}(B_-)$ be the open derived subscheme of $\mathbb{R}\op{Spec}(B_-)$ corresponding to $U$ in \ref{204924_6Jan22}.
Note that $\tilde{U}$ is $P$-invariant. 
So, from the weak equivalence of \ref{163437_16Jan22}, we define the corresponding subfuntor of $\mathcal{F}_-$. 
 We denote it by $\mathcal{F}_{(\alpha_1, \cdots, \alpha_s)}^\op{HN}$.
\label{165325_16Jan22}

\end{dfn}

\begin{dfn}[{\cite[Definition 1.8, 1.13 and Lemma 1.12]{pridham_2012representability}}]
 Let $\mathcal{G}: \op{dg}_{\leq 0}\op{Alg}_k \rightarrow \op{SSet}$ be a homotopy-preserving and homotopy-homogeneous functor (e.g. if $\mathcal{G}$ is a derived geometric $n$-stack).
For any $C \in \op{dg}_{\leq 0}\op{Alg}_k$, let $x \in \mathcal{G}(C)$ be a point.
We define the tangent functor 
\[\xymatrix@=5pt{
T_x(\mathcal{G}/k) & : & \op{dg}_{\leq 0}\op{Mod}_{C} \ar[r] & \op{SSet} \\
&& M \ar@{|->}[r] & T_x(\mathcal{G}/k)(M)& \hspace{-9pt} :=  \mathcal{G}(C \oplus M) \times_{\mathcal{G}(C)}^h \{x\}.
}\]
We also define 
\[
 \op{D}_x^{n-i}(\mathcal{G},M) := \pi_i(T_x(\mathcal{G}/k)(M[-n])).
\]
Note that $\op{D}^{n-i}_x$ is not dependent on $n,i$ in the above defnition.
\end{dfn}

\begin{rmk}
\label{212436_23Jan22}
We have the equvalence of (non-derived) stacks
\begin{equation}
  \mathcal{FC}\op{oh}^{\op{HN}}_{(\alpha_1, \cdots, \alpha_s)}(X) \simeq \pi_0[\tilde{U}/ P]
\end{equation}
by Definition \ref{165325_16Jan22} and (1) of Theorem \ref{160827_17Jan22}, where $\pi_0[\tilde{U}/ P]$ is the non-derived stack associated to $[\tilde{U}/ P]$.

 From Corollary \ref{160735_17Jan22}, (2) of Theorem \ref{160827_17Jan22} and Proposition \ref{163437_16Jan22}, We have 
\begin{equation}
  H^iT^{\bullet}_{[\mathscr{F}]}\mathcal{RFC}\op{oh}^{\op{HN}}_{(\alpha_1, \cdots, \alpha_s)}(X) \simeq \op{D}^i_{\Gamma^{\text{fil}}_{[p,q]}(\mathscr{F})}([\tilde{U}/ P], k).
\end{equation}
So, our translation of $\mathcal{RFC}\op{oh}^{\op{HN}}_{(\alpha_1, \cdots, \alpha_s)}(X)$ into $[\tilde{U}/ P]$ (and $(\bar{W} \circ N ) (\underline{\mathcal{F}_{(\alpha_1, \cdots, \alpha_s)}^\op{HN})})$ is suitable.
\end{rmk}

\subsection{Another construction of deived moduli of HN-filtrations by methods by Di Natale and Pridham}
\label{164538_16Jan22}

Here, we construct derived moduli stacks of HN-filtrations by using the method in \cite{di2017derived}.
From now, $X$ means a smooth projective variety over $k$.

\begin{dfn}
Let $C$ be a $k$-algebra.
 We define $\op{QCoh}(X \times_k \op{Spec}(C))_{\text{flat}}$ (resp. $\op{FQCoh}(X \times_k \op{Spec}(C))_{\text{flat}}$) to be tha category of (resp. filtered) quasi-coherent sheaves on $X \times_k \op{Spec}(C)$ which are (resp. whose associated graded sheaves) are flat over $\op{Spec}(C)$.
\end{dfn}

\begin{dfn}
 Let $O^{\bullet}$ be a cosimplicial $k$-algebra. 
We define $\op{Mod}(O^\bullet)$ (resp. $\op{FMod}(O^\bullet)$) to be the category of cosimplicial $O^{\bullet}$-modules (resp. filtered cosimplicial $O^\bullet$-modules).
We also define $\op{Mod}_{\text{cart}}(O^\bullet)$ (resp. $\op{FMod}_{\text{cart}}(O^\bullet)$ ) to be the category of cosimplicial cartesian $O^\bullet$-modules (resp. filtered cosimplicial cartesian $O^\bullet$-modules).

We define $\op{dgMod}(O^\bullet), \op{dgFMod}(O^\bullet),\op{dgMod}_{\text{cart}}(O^\bullet), \op{dgFMod}_{\text{cart}}(O^\bullet)$ in the same way.
\end{dfn}

\begin{egg}
\label{173031_12Jan22}
Let $ \{U_i\}_{i \in I}$ be afinite open affine covering of $X$ and let $\mathcal{U} := \coprod_i U_i$.
We define a simplicial scheme $\mathcal{U}^\bullet$ and a cosimplicial $k$-algebra $\mathcal{O}(X \times_k \op{Spec}(C))^\bullet$ as follows; 
\begin{align*}
  \mathcal{U}_n := \overbrace{\mathcal{U} \times_X \mathcal{U} \times_X \cdots \times_X \mathcal{U}}^{(n+1)-\text{times}} &= \coprod_{i_0, \cdots, i_n}U_{i_0 \cdots i_n} = \coprod_{i_0, \cdots, i_n} U_{i_0} 
\times_X \cdots \times_X U_{i_n}, \\
O(X \times_k \op{Spec}(C))^n &:= \Gamma \left(\mathcal{U}_n \times_k \op{Spec}(C), \mathscr{O}_{\mathcal{U}_n \times_k \op{Spec}(C)} \right)
\end{align*}
, where $C$ is a $k$-algebra.
Note that $\op{dgMod}(O(X \times_k \op{Spec}(C))^\bullet)$ and $\op{dgFMod}(O(X \times_k \op{Spec}(C))^\bullet)$ have model structures (\cite[page 835, 844]{di2017derived}).
Because $X \times_k \op{Spec}(C)$ is quasi-compact and semi-separated, we have \[
	\op{Ho}(\op{dgMod_{\op{cart}}}(O(X \times_k \op{Spec}(C))^\bullet)) \simeq D(\op{QCoh}(X \times_k \op{Spec}(C))) \quad (\text{\cite[Thm 5.5.1]{huttemann2010derived}}).
\]			
In particular, we have a equivalence
\begin{equation}
  \op{Mod}_{\op{cart}}(O(X \times_k \op{Spec}(C))^\bullet) \simeq \op{Qcoh}(X \times_k \op{Spec}(C)) 
\end{equation}
, where the correspondence from the left-hand (resp.right-hand) side to the right-hand (resp.left-hand) side is obtained by the sheafification functor (resp.the global section functor).
Since the sheafification functor and the global section functor preserve inclusions, we also have a equivalence
\begin{equation}
 \op{FMod}_{\op{cart}}(O(X \times_k \op{Spec}(C))^\bullet) \simeq \op{FQCoh}(X \times_k \op{Spec}(C)) .\label{164301_15Jan22}
\end{equation}
\end{egg}


Using approaches by Di Natale and Pridham and Example \ref{173031_12Jan22}, we obtain another construction of derived moduli stacks of HN-filtrations.
We first consider the functor 
\[\xymatrix@R=5pt@C=5pt{
 \mathcal{FM}{od}_{\text{cart}}(O(X \times_k -))_{\text{flat}}& :& \op{Alg}_k \ar[r]  & \op{sCat} \\ 
C  \ar@{}[r]^(.20){}="a"^(.90){}="b" \ar@{|->} "a";"b"&&& \hspace{-1.5cm}\op{FMod}_{\text{cart}}(O(X \times_k \op{Spec}(C))^\bullet)_{\text{flat}}.
}
\]
We also define $\mathcal{F}dg\mathcal{M}{od}_{\text{cart}}(O(X \times_k -))$ in the same way .
By the flatness of the objects in $\op{FMod}_{\text{cart}}(O(X \times_k \op{Spec}(C))^\bullet)_{\text{flat}}$, the functor  $\mathcal{FM}\op{od}_{\text{cart}}(O(X) \times_k -)_{\text{flat}}$ is embedded in the functor $\mathcal{F}dg\mathcal{M}{od}_{\text{cart}}(O(X \times_k -))$ and this is an open embedding.

When we denote the subfunctor of $\mathcal{FM}{od}_{\text{cart}}(O(X \times_k -))_{\text{flat}}$ which parametrizes coherent sheaves by $\mathcal{FM}{od}_{\text{coh}}(O(X \times_k -))_{\text{flat}}$, we get a open embedding $\mathcal{FM}{od}_{\text{coh}}(O(X \times_k -))_{\text{flat}} \hookrightarrow \mathcal{FM}{od}_{\text{cart}}(O(X \times_k -))_{\text{flat}}$.

Let $\alpha_1, \alpha_2,\cdots, \alpha_s \in \mathbb{Q}[t]$ such that $\alpha_1 \succ \alpha_2 \succ \cdots \succ \alpha_s$. 
We define the functor $\mathcal{FM}od_{\text{coh}, (\alpha_1, \cdots, \alpha_s)}^{\text{HN}}$ $(O(X \times_k -))_{\text{flat}}$ by using the equivalence \ref{164301_15Jan22} as follows: 
\[
 \xymatrix@R=5pt@C=0pt{ \mathcal{FM}od_{\text{coh}, (\alpha_1, \cdots, \alpha_s)}^{\text{HN}}(O(X \times_k -))_{\text{flat}}& :& \op{Alg}_k \ar[r]  &  
\op{sCat}  \\ 
\hspace{-1.5cm} C \ar@{}[r]^(-.15){}="a"^(.55){}="b" \ar@{|->} "a";"b" &&&
  \hspace{-2.5cm} \left\{\small {\begin{gathered} \mathcal{M} \in \op{FMod}_{\text{cart}}(O(X \times_k \op{Spec}(C))^\bullet)_{\text{flat}} \text{ such that } \\
	\text{the corresponding object in } 
\op{FQcoh}(X \times_k \op{Spec}(C)) \\ 
\text {is an object in } \mathcal{FC}\op{oh}^{\text{HN}}_{(\alpha_1, \cdots, \alpha_s)}(X)
\end{gathered}} \right\}.
}\]
We have an embedding $\mathcal{FM}od_{\text{coh}, (\alpha_1, \cdots, \alpha_s)}^{\text{HN}}(O(X \times_k -))_{\text{flat}} \hookrightarrow \mathcal{FM}{od}_{\text{coh}}(O(X \times_k -))_{\text{flat}}$ because we have an open immersion of moduli stacks $\mathcal{FC}\op{oh}^{\text{HN}}_{(\alpha_1, \cdots, \alpha_s)}(X) \hookrightarrow \mathcal{FC}\op{oh}(X)$.
In addition, when let $\mathbf{M}^0_{\text{filt}}$ be the functor defined in \cite[page 849]{di2017derived}, we have an open embedding $\mathcal{FM}od_{\text{coh}, (\alpha_1, \cdots, \alpha_s)}^{\text{HN}}(O(X \times_k -))_{\text{flat}} \hookrightarrow \mathbf{M}^0_\text{filt}$ and $\mathbf{M}^0_\text{filt}$ satisfies the conditions of \cite[Cor 3.32]{di2017derived}.
Moreover, the corresponding groupoid-valued functor $\pi_0\mathcal{FM}od_{\text{coh}, (\alpha_1, \cdots, \alpha_s)}^{\text{HN}}(O(X \times_k -))_{\text{flat}}$ is a stack because $\mathcal{FC}\op{oh}^{\text{HN}}_{(\alpha_1, \cdots, \alpha_s)}(X)$ is a stack.
Thus, we obtain the following proposition.

\begin{prop}
 The functor $\mathcal{FM}od_{\text{coh}, (\alpha_1, \cdots, \alpha_s)}^{\text{HN}}(O(X \times_k -))_{\text{flat}}$ satisfies the condtions of \cite[Cor 3.32]{di2017derived}.
To be more precise, we consider the functor 
\[\xymatrix@R=5pt@C=0pt{
 \widebreve{\mathcal{FM}od}_{\text{coh}, (\alpha_1, \cdots, \alpha_s)}^{\text{HN}}(O(X \times_k -))_{\text{flat}}&  : & \op{dg}_{\leq 0}\op{Nil}_{k}^{\flat} \ar[r] &\op{sCat} \\
 \hspace{-3cm} C \ar@{}[r]^(-.45){}="a"^(.10){}="b" \ar@{|->} "a";"b" &&&
 \hspace{-4.5cm}\left\{ \small {\begin{gathered}
			      \mathcal{M} \in \op{FdgMod}_{\text{cart}}(O(X \times_k \op{Spec}(C))^\bullet) \text{ such that}\\ 
\mathcal{M} \text{ is cofibrant and } 
\mathcal{M} \otimes_C H^0(C) 
\text{ is weakly equivalent to } \\
\text{an object in } 
\mathcal{FM}od_{\text{coh}, (\alpha_1, \cdots, \alpha_s)}^{\text{HN}}(O(X \times_k \op{Spec}(H^0(C)))^\bullet)_{\text{flat}}.
     \end{gathered}} \right\}.
}\]
Then, the simplicial set-valued functor $(\bar{W} \circ N)\widebreve{\mathcal{FM}od}_{\text{coh}, (\alpha_1, \cdots, \alpha_s)}^{\text{HN}}(O(X \times_k -))_{\text{flat}}$ is a derived geometric 1-stack.
\label{165404_16Jan22}
\end{prop}

\subsection{Comparing two construction of derived moduli spaces of HN-filtrations}

We compare two constructions in Subsection \ref{164506_16Jan22} and \ref{164538_16Jan22}. 
We need the following lemma to do it.

\begin{lem}[{\cite[Lemma 6.24]{pridham2021introduction}}]
 Let $\Phi: \mathcal{Y} \rightarrow \mathcal{Z}$ be a morphism of $n$-geometric derived stacks.
Then, $\Phi$ is a weak equivalence if and only if 
\begin{enumerate}
 \item $\Phi(C) : \mathcal{Y}(C) \rightarrow \mathcal{Z}(C)$ is a weak equivalence for any $k$-algebra $C$, and
\item for all $k$-algebra $C$, all $C$-module $N$ and all $y \in \mathcal{Y}(C)$, the maps $D^i_y(\mathcal{Y}, N) \rightarrow D^i_{\Phi x}(\mathcal{Z}, N)$ are isomorphisms for all $i > 0$. 
\end{enumerate}
\label{212143_23Jan22}
\end{lem}

\begin{thm}
The derived stacks $[\tilde{U}/P]$ in Definition \ref{165325_16Jan22} and $(\bar{W} \circ N)\widebreve{\mathcal{FM}od}_{\text{coh}, (\alpha_1, \cdots, \alpha_s)}^{\text{HN}}(O(X \times_k -))_{\text{flat}}$ in Proposition \ref{165404_16Jan22} are weakly equivalent.
\label{232359_30Jan22}
\end{thm}

\begin{proof}
\item
\subsubsection*{Step 1}
 From Proposition \ref{163437_16Jan22}, it is enough to prove $(\bar{W} \circ  N) (\underline{\mathcal{F}_{(\alpha_1, \cdots, \alpha_s)}^\op{HN}})$ and $(\bar{W} \circ N) \widebreve{\mathcal{FM}od}_{\text{coh}, (\alpha_1, \cdots, \alpha_s)}^{\text{HN}}$ $(O(X \times_k -))_{\text{flat}}$ are weakly equivalent.
To construct a morphism between them, we construct a morphism from $\underline{\mathcal{F}_{(\alpha_1, \cdots, \alpha_s)}^\op{HN}}$ to $\widebreve{\mathcal{FM}od}_{\text{coh}, (\alpha_1, \cdots, \alpha_s)}^{\text{HN}}(O(X \times_k -)z)_{\text{flat}}$.
For all $C \in \op{dg}_{\leq 0}\op{Nil}_{k}^{\flat}$, objects of $\mathcal{F}_{(\alpha_1, \cdots, \alpha_s)}^\op{HN}(C)$ are given by the triples $(C \rightarrow C', \omega ,  g)$, where 
\begin{enumerate}
 \item $C \rightarrow C':$ a strict \'{e}tale covering (\cite[Definition 2.17]{pridham2013constructing}),
\item $\omega$ is an object in $\mathcal{F}_{(\alpha_1, \cdots, \alpha_s)}^\op{HN}(C')$,  
\item $g \in P \otimes_k (C'^0 \otimes_{C^0} C'^0)$ which satisfies $g \cdot \op{pr}_1^*\omega = \op{pr}_0^*\omega $ and the cocycle condtition $\op{pr}_{01}^*g \circ \op{pr}_{12}^*g = \op{pr}_{02}^*g$.
\end{enumerate}
In other words, an object in $\mathcal{F}_{(\alpha_1, \cdots, \alpha_s)}^\op{HN}(C)$ is a filtered graded locally free $C^0$-module $M$ with a filtered graded unital $C$-linear $A_{\infty}$-action $\mu$ of $A \otimes_k C$ on $M \otimes_{C^0} C $ such that $H^0(M \otimes_{C^0} C)$ is an object  in $\mathcal{FC}\op{oh}^{\op{HN}}_{(\alpha_1, \cdots, \alpha_s)}(X)$ (modulo truncation as graded modules).

\subsubsection*{Step 2}

The condition $M$ has an $A_{\infty}$-action of $A \otimes_k C$ induces a dg-action $\mu'$ of $A \otimes_k C$ on $\op{Bar}_{A \otimes_k C} (M \otimes_{C^0} C)$ (\cite[Proposition 3.4.9]{ciocan2001}).

 We define a dg sheaf $\tilde{\mathscr{S}}(M \otimes_{C^0}C)$ on $(X \times \op{Spec}(C^0), \mathcal{R}_C)$ to be a chain complex of coherent sheaves $\widetilde{\op{Bar}_{A \otimes_k C} (M \otimes_{C^0} C)}$, which has a structure of a dg sheaf on $(X \times \op{Spec}(C^0), \mathcal{R}_C)$, where $\mathcal{R}_C := \bigoplus_i (\widetilde{A \otimes C_i})$ is a sheaf of dg-algebras on $X \times \op{Spec}(C^0)$.
Note that $\widetilde{\op{Bar}_{A \otimes_k C} (M \otimes_{C^0} C)}$ is independent of the choice of $M$.

$\mathcal{H}^0(\tilde{\mathscr{S}}(M \otimes_{C^0}C) \otimes \mathscr{O}_{X \times_k \op{Spec}(H^0(C))})$ is an object in $\mathcal{FC}\op{oh}^{\op{HN}}_{(\alpha_1, \cdots, \alpha_s)}(X)$ from the construction.
Here, $\mathcal{H}^0$ is the functor taking $0$-th cohomology of cochain complexes of sheaves.
We also have $\mathcal{H}^0(\tilde{\mathscr{S}}(M \otimes_{C^0}C) \otimes \mathscr{O}_{X \times_k \op{Spec}(H^0(C))})$ is flat over $\op{Spec}(H^0(C))$ because of the locally freeness of $M$.

Then, we think of $\tilde{\mathscr{S}}(M \otimes_{C^0}C)$ as an object in $\op{dgFMod}(O(X \times_k \op{Spec}(C))^\bullet)$ by taking the global section functor.
Thus, from the above discussion, we have a morphism of simplicial category-valued functors from $\underline{\mathcal{F}_{(\alpha_1, \cdots, \alpha_s)}^\op{HN}}$ to $\widebreve{\mathcal{FM}od}_{\text{coh}, (\alpha_1, \cdots, \alpha_s)}^{\text{HN}}(O(X \times_k -))_{\text{flat}}$.


After that, we take the functor $\bar{W} \circ N$ and obtain a functor between derived stacks
\begin{equation}
 \Psi : (\bar{W} \circ  N) (\underline{\mathcal{F}_{(\alpha_1, \cdots, \alpha_s)}^\op{HN}}) \rightarrow (\bar{W} \circ N) \widebreve{\mathcal{FM}od}_{\text{coh}, (\alpha_1, \cdots, \alpha_s)}^{\text{HN}}(O(X \times_k -))_{\text{flat}}.
\end{equation}

\subsubsection*{Step 3}

Next, we see that $\Psi$ satisfies the conditions of Lemma \ref{212143_23Jan22}.
We have $\Psi$ satisfies the condition $(1)$ of the lemma from Remark \ref{212436_23Jan22} and Theorem \ref{150320_6Jul21}.
As for the condition (2), for any object $M \in (\bar{W} \circ  N) (\underline{\mathcal{F}_{(\alpha_1, \cdots, \alpha_s)}^\op{HN}})(C)$, we have an object $\mathscr{F} \in \mathcal{FC}\op{oh}^{\op{HN}}_{(\alpha_1, \cdots, \alpha_s)}(X)$ such that $M \simeq \Gamma^{\text{fil}}_{[p,q]}(\mathscr{F})$ and $\op{Ext}^i_{-}(\mathscr{F}, \mathscr{F} \otimes^{\mathbb{L}}_C N) \simeq D^i_{M}(\bar{W} \circ  N) (\underline{\mathcal{F}_{(\alpha_1, \cdots, \alpha_s)}^\op{HN}}),N)$.
And, for any $\mathscr{F} \in \mathcal{FC}\op{oh}^{\op{HN}}_{(\alpha_1, \cdots, \alpha_s)}(X)$, we can take a filtered locally free resolution $\mathscr{E}^\bullet \rightarrow \mathscr{F} \rightarrow 0$ and calculate $\op{Ext}^i_{-}(\mathscr{F}, \mathscr{F} \otimes^{\mathbb{L}}_C N)$ by this resolution.
Moreover, in order to calculate $D^i_{\Psi(M)}((\bar{W} \circ N) \widebreve{\mathcal{FM}od}_{\text{coh}, (\alpha_1, \cdots, \alpha_s)}^{\text{HN}}(O(X \times_k -))_{\text{flat}},N)$, we can take a flat resolution of $\Psi(M)$ instead of a cofibrant resolution (this is because derived functors are calculated by their deformations. in detail, see \cite[Prop 3.4]{shulman2009homotopy})). 
Thus, we have the condition (2) from the equivalence \ref{164301_15Jan22} and the fact that $\mathscr{E}^\bullet \rightarrow \mathscr{F} \rightarrow 0$ is mapped to a flat resolution by this equivalence.
Therefore, $\Psi$ is a weak equivalence and this completes our proof.
\end{proof}

\begin{rmk}
 We can carry the same argument in the case of derived moduli stacks of (non-filtered) semistable sheaves: first, we translate the dg-stacks of semistable sheaves constructed in \cite{behrend2014} into derived stacks in the sense of To\"{e}n, on the other hand, we construct derived moduli stacks of semistable sheaves by using techniques of Pridham and Di Natale, finally, we can show weak equivalence between them.
\end{rmk}

\begin{egg}
\label{132127_10Feb22}
We provide a simple example of Lagrangian morphisms related to derived moduli stacks of HN-filtrations.
Let $X$ be a projective Calabi-Yau variety of dimension $d$.
Then, the derived stack $\mathcal{RC}oh(X)$ of coherent sheaves on $X$ is equipped with a $(2-d)$-shifted symplectic structure (for example, see \cite[Section 5.3]{toen2014derived}).
We also have the derived moduli of filtered coherent sheaves on $X$ is weakly equivarent to the derived moduli of sequences of morphisms in the dg-category $\op{Coh}_{dg}(X)$ of coherent sheaves on $X$ (cf. \cite[Remark 3.6]{di2019global}).
Note that derived moduli of sequences of morphisms of coherent sheaves are defined as in the case of filtered sheaves by using a model structure on the category of morphisms in a stable model category which is constructed in \cite{gwilliam2018enhancing}.
Moreover, we can show that the derived moduli of sequences of morphisms of length 1 in $\op{Coh}_{dg}(X)$ is weakly equivalent to the derived moduli of morphisms in $\op{Coh}_{dg}(X)$ which is defined in \cite[Definition 3.18]{toen2007moduli} by comparing their tangent complexes.
On the other hand, for any dg-category $T$ over $k$ which is smooth, proper and equipped with an orientation of dimension $d$, the derived moduli stack $\mathcal{M}^{(1)}_T$ of morphisms in $T$ has a correspondence 
\[
 \mathcal{M}_T \times \mathcal{M}_T \overset{s,c}{\longleftarrow} \mathcal{M}^{(1)}_T \overset{t}{\longrightarrow} \mathcal{M}_T,
\]
,where $s,c,t$ send morphisms in $T$ to their sources, cones and targets respectively.
This induces a Lagrangian structure with respect to the $(2-d)$-shifted symplectic structure on $\mathcal{M}^3_T$ (cf. \cite[Section 5.3]{toen2014derived}, \cite[Corollary 6.5]{brav2021relative}).
So, considering the above argument and that derived moduli of HN-filtrations are open substacks of derived moduli of filtered sheaves, we have morphisms from the derived moduli stacks of length 2 HN-filtrations to $\mathcal{RC}oh(X)^3$ which carries Lagrangian structures with respect to the $(2-d)$-shifted symplectic sturucture on $\mathcal{RC}oh(X)^3$.
This gives a symplectic geometrical interpretation of moduli of unstable sheaves on $X$.
\end{egg}

\subsection*{Acknowledgements}
The auther would like to thank his adviser Professor Hajime Kaji for his constant support.
He is also grateful to Professor Ryo Ohkawa, Professor Isamu Iwanari for useful conversations.
Finally, he is so thankful to the referee for pointing out mistakes of the manuscript.
This work is supported by Grant-in-Aid for JSPS Fellows (Grant Number 22J11405)


\printbibliography


\end{document}